\documentclass[12pt,a4paper]{article}
\usepackage{amsmath,amsthm,amssymb,lineno}
\usepackage{txfonts,pxfonts,upgreek} 
\usepackage{titlesec}
\usepackage{caption}
\usepackage{subcaption}
\titlespacing*{\section}
{0pt}{5.5ex plus 1ex minus .2ex}{4.3ex plus .2ex}
\titlespacing*{\subsection}
{0pt}{5.5ex plus 1ex minus .2ex}{4.3ex plus .2ex}
\usepackage{amsmath,  amsthm, amssymb, color}
\usepackage{enumerate, hyperref}
\usepackage{enumitem}

\usepackage{graphicx}
\usepackage{tikz}
\usetikzlibrary{bbox}

\usepackage[open,openlevel=2,atend]{bookmark}

\usepackage[abbrev,msc-links,nobysame]{amsrefs}  
\usepackage{doi}

\renewcommand{\PrintDOI}[1]{\doi{#1}}

\usepackage{pifont}

\usepackage{amsgen,amstext,amsbsy,amsopn}
\usepackage{float}
\usepackage{enumitem}
\usepackage{xcolor}
\newcommand\red[1] {{\color{red} #1}}

\usepackage{booktabs}
\usetikzlibrary{intersections,calc,arrows.meta}

\newtheorem{thm}{Theorem}[section]
\newtheorem{obs}{Observation}[section]

\newtheorem{prop}{Proposition}[section]
\newtheorem{lem}{Lemma}[section]
\newtheorem{cor}{Corollary}
[section]

\newtheorem{prob}{Problem}[section]
\newtheorem{conj}[prob]{Conjecture}
\newtheorem{claim}{Claim}[section]

\newtheorem{remark}{Remark}[section]

\newtheorem{defi}{Definition}[section]
\usepackage{hyperref}
\hypersetup{
    colorlinks=true,
    linkcolor=blue,
    citecolor=blue,
    urlcolor=blue
}

\usetikzlibrary{backgrounds}
\usetikzlibrary{arrows}
\usetikzlibrary{shapes,shapes.geometric,shapes.misc}
\pgfdeclarelayer{edgelayer}
\pgfdeclarelayer{nodelayer}
\pgfsetlayers{background,edgelayer,nodelayer,main}
\tikzstyle{none}=[inner sep=0mm]
\tikzstyle{blacknode}=[fill=black, draw=black, shape=circle, minimum
size=0.20cm, inner sep=0pt]
\tikzstyle{whitenode}=[fill={rgb,255: red,245; green,245; blue,245},
draw=black, shape=circle, minimum size=0.20cm, inner sep=1pt]
\tikzstyle{whitenode_v1}=[fill={rgb,255: red,245; green,245; blue,245},
draw=black, shape=circle, minimum size=0.4cm, inner sep=0.8pt, scale=1]
\tikzstyle{blacknode_v1}=[fill=black, draw=black, shape=circle, minimum
size=0.1cm, inner sep=0pt]
\tikzstyle{bluenode}=[fill=blue, draw=black, shape=circle, minimum
size=0.1cm, inner sep=0pt]
\tikzstyle{black_bold}=[-, draw=black, line width=0.6mm]
\tikzstyle{blackedge}=[-, draw=black, fill=none, line width=0.3mm]
\tikzstyle{rededge}=[-, line width=0.3mm, draw=red]
\tikzstyle{blueedge}=[-, line width=0.3mm, draw=blue]
\tikzstyle{grayedge}=[-,line width=0.2mm, draw={rgb,255: red,64; green,64; blue,64}]

\tikzstyle{shadow_silver}=[-, draw=black, fill={rgb,255: red,186; green,186; blue,186}, line width=0.45mm]

\definecolor{xwhite}{RGB}{245,245,245}%
\definecolor{xback}{RGB}{0,0,0}%
\definecolor{xblue}{RGB}{0,0,139}%
\definecolor{xgreen}{RGB}{50,205,50}%
\definecolor{xcrimson}{RGB}{220,20,60}%
\definecolor{xgold}{RGB}{255,215,0}
\definecolor{xmoccasin}{RGB}{255,228,181}

\usepackage{lineno} 
\usepackage{listings}
\usepackage{xcolor}

\usepackage{multirow,multicol}

\newcommand\equ[2]
{
\begin{equation}\label{#1}
#2
\end{equation}
}

\newcommand\Eqnn[1]
{
	\begin{eqnarray*}
		#1
	\end{eqnarray*}
}

\newcommand\Eqn[2]
{
	\begin{eqnarray}\label{#1}
		#2
	\end{eqnarray}
}

\newcounter{countcase}

\def \P {{\cal P}}
\def \F {\mathcal{F}}
\def \C {\mathcal{C}}

\newcommand\maxe[2]
{
ex_{\P_1,#1}(K_{#2})
}

\newcommand\Maxe[2]
{
\text{Forb}_{\P_1,#1}(K_{#2})
}

\newcommand\Floor[1]
{
\left \lfloor  #1 \right \rfloor 
}

\def \Fvertex {V_{\mathrm{fake}}}

\begin{document}

\title{\textbf{Extremal  1-planar graphs without $\boldsymbol{k}$-cliques }}

\author{
  Licheng Zhang \thanks{E-mail: \texttt{lczhangmath@163.com}.  College of Mathematics and Statistics, Hunan  Normal  University}
\and
 Yuanqiu Huang \thanks{Corresponding author. E-mail: \texttt{hyqq@hunnu.edu.cn}. College of Mathematics and Statistics, Hunan  Normal  University}
 \and
 Fengming Dong \thanks{E-mail: \texttt{ fengming.dong@nie.edu.sg
 and  donggraph@163.com}.
 National Institute of Education, 
 Nanyang Technological University, Singapore}}
\date{}

\maketitle

\noindent
\begin{abstract}
In 2016, Dowden initiated the study of planar Tur\'an-type problems, which has since attracted considerable attention.  
Recently, Bekos et al. proved that every $K_3$-free $1$-planar graph on $n\ge 4$ vertices has at most $3n-6$ edges. 
In this paper, we strengthen this bound to $3n - 8$, which is tight for all even $n \ge 8$.
Furthermore, we show that every $K_4$-free $1$-planar graph on $n \ge 3$ vertices has at most $\bigl\lfloor \tfrac{7n}{2} \bigr\rfloor - 7$ edges, and this bound is tight for all integers $n \ge 9$. We also prove that every $K_5$-free $1$-planar graph on $n \ge 3$ vertices has at most $4n - 8$ edges, which is tight for $n = 8$ and for all integers $n \ge 10$.

\end{abstract}



\section{Introduction}

Determining the largest possible size in a class of graphs
 is a basic problem in extremal graph theory. 
 A graph is \emph{$H$-free} if it contains no subgraph isomorphic to a graph $H$.
A well-known result is Tur\'an's Theorem \cite{MR0018405}, 
which gives an exact upper bound on the number of edges in a $K_k$-free 
$n$-vertex graph.
Tur\'an's theorem has led to  deep developments, including the Erd\H{o}s--Stone theorem~\cite{MR0018807}, generalized Tur\'an-type  problems \cite{MR5037331}, and Tur\'an-type  problems for hypergraphs \cite{MR2866732}.

Tur\'an-type extremal problems for various families of planar graphs have been widely investigated. Notable examples include $K_r$-free planar graphs for $r \in \{3,4\}$, $C_r$-free planar graphs for $r \ge 4$ (see \cite{MR3549506, MR4474377, gyori2023, MR4357025, MR4828036, MR4866462}), and planar graphs excluding theta graphs (see Lan, Shi, and Song \cite{MR3990020}).

It is natural to extend the study of Tur\'an-type extremal problems to 
$k$-planar graphs, that is, graphs that can be drawn in the plane such that each edge is crossed at most k times.
In 2025,
Bekos et al.
\cite{Bekos2025} 
initiated the study of Tur\'an-type problems on $k$-planar graphs forbidding  short cycles. 
It is known that  for $k \ge 2$, any $k$-planar graph on $n \ge 3$ vertices has at most $3.81\sqrt{k}\,n$ edges \cite{MR4010251}.  
Bekos et al. \cite{Bekos2025} showed that every $C_3$-free $k$-planar graph on $n \ge 2$ vertices
has its size $m$ at most 
$3.182\sqrt{k}\,n$ 
whenever $m \ge \frac{9}{2}n$,
and
every $C_4$-free $k$-planar graph on $n\ge 2$ vertices 
has its size $m$ at most 
$3.016\sqrt{k}\,n$
whenever $m\ge 3.483n$.

Restricted to Tur\'{a}n-type extremal problems for $1$-planar graphs, it is known that every $1$-planar graph on $n \ge 3$ vertices has at most $4n - 8$ edges \cite{MR1606052}. Based on a recently established density formula \cite{MR4821240}, Bekos et al.~\cite{Bekos2025} proved that every  
$K_3$-free $1$-planar graph on $n \ge 4$ vertices has at most $3n - 6$ edges.
Furthermore, using the discharging method, Bekos et al.~\cite{Bekos2025} showed that every $C_4$-free $1$-planar graph on $n$ vertices has at most $\frac{5}{2}n - 5$ edges,
and
 every $\{C_3, C_4\}$-free $1$-planar graph  on $n$
 vertices 
 has at most $\frac{12}{5}n$ edges.
In addition, Xu and Chang~\cite{XuChang25} established spectral Turán-type theorems for $K_k$-free $1$-planar graphs.

In this article, we 
establish tight upper bounds for the size of 
$K_k$-free $1$-planar graphs for $k=3,4,5$.


\subsection{Main results}

For any graph $G$,
let $V(G), E(G)$,
$v(G):=|V(G)|$ and $e(G):=|E(G)|$
be its vertex set,
edge set, 
order and size,
respectively.
For any positive integers $n$ 
and $k$, 
let $\Maxe{n}{k}$ be the set of 
$K_k$-free $1$-planar graphs $G$ of order $n$, and 
let 
$\maxe{n}{k}$
be the maximum value of $e(G)$ 
over all graphs $G$ in  $\Maxe{n}{k}$.
Bekos et al.~\cite{Bekos2025} showed that 
$\maxe{n}{3}\le 3n-6$.
In this paper, we will improve this bound to $\maxe{n}{3}\le 3n-8$ for $n\ge 4$.
Due to a construction by Karpov \cite{zbMATH06347737}, 
$\maxe{n}{3}\ge 3n-8$ holds for all 
even $n\ge 8$. Hence we obtain
 the following conclusion.

\begin{thm}\label{thm:K3}
$\maxe{n}{3}\le 3n-8$ for $n\ge 4$, where the equality holds for all even $n\ge 8$.
\end{thm}

%
%

We then study $\maxe{n}{k}$
for $k\ge 4$. 

\begin{thm}\label{thm:K4}

$\maxe{n}{4}= \Bigl\lfloor \frac{7n}{2}\Bigr\rfloor-7$
for all $n\ge 9$,
and 
$\maxe{n}{4}=\bigl\lfloor \frac{n^2}3\bigr\rfloor
-\bigl\lfloor \frac{n}8\bigr\rfloor
$ for $1\le n\le 8$.

\end{thm}

Since bipartite graphs are $K_3$-free and tripartite graphs are $K_4$-free, from Theorems~\ref{thm:K3} and~\ref{thm:K4} we obtain the following corollaries, respectively.

\begin{cor}[Karpov \cite{zbMATH06347737}]\label{cor:bipartite}
Let $G$ be a bipartite $1$-planar graph of order $n\ge 4$. Then
$
e(G)\le 3n-8.
$
\end{cor}

\begin{cor}[Suzuki \cite{MR4305146}]\label{cor:tripartite}
Let $G$ be a tripartite $1$-planar graph of order $n\ge 3$. Then
$
e(G)\le \frac{7}{2}n-7.
$
\end{cor}
As a side remark, we also obtain the following easy consequences. 
These results are known and were previously proved by discharging arguments, here they follow immediately from the edge bounds.

\begin{cor}[Fabrici and Madaras \cite{MR2297168}]\label{cor:delta6-triangle}
Every $1$-planar graph $G$ with minimum degree 6 contains a $K_3$.
\end{cor}

\begin{cor}[Hud\'ak and Madaras \cite{MR2802062}]\label{cor:delta7-K4}
Every $1$-planar graph $G$ with minimum degree 7 contains a $K_4$.
\end{cor}
Indeed, if $\delta(G)=6$, then 
$
e(G)\ge 3n.
$
Therefore, by Theorem~\ref{thm:K3}, $G$ contains a $K_3$.
Similarly, if $\delta(G)= 7$, then
$
e(G)\ge \frac{7}{2}n.
$
Therefore, by Theorem~\ref{thm:K4},  $G$ contains a $K_4$.

Note that  Nakamoto, Noguchi,  and Ozeki \cite{MR3502764} showed that there exist infinitely many 4-colorable optimal 1-planar graphs. It is obvious that these graphs are also $K_5$-free, but we do not see the explicit constructions in \cite{MR3502764}. In this paper, we explicitly construct $K_5$-free graphs for orders $n=8$ and $n\ge 10$ (see Proposition \ref{prop:opc} for details). We therefore have the following theorem.

\begin{thm}\label{thm:K5}
$\maxe{n}{5}= 4n-8$
	for $n=8$ or $n\ge 10$,
	and
	$\maxe{n}{5}
	=\left \lfloor \frac{3n^2}{8}\right \rfloor 
	-3\left \lfloor 
	\frac{n}{9}\right \rfloor
	$ for $1\le n\le 7$ or $n=9$.

\end{thm}

%
%
%
%
%

\begin{remark}
Since $K_5$ is a subgraph of $K_6$, every $K_5$-free graph is also $K_6$-free. Hence, the upper bounds in Theorem~\ref{thm:K5} also apply to $K_6$-free $1$-planar graphs, except possibly for a few small values of $n$. Furthermore, since $K_7$ is not $1$-planar~\cite{MR2876333}, every $1$-planar graph is automatically $K_7$-free. Thus, for every $k\ge 7$, the condition of being $K_k$-free imposes no additional restriction on $1$-planar graphs, and the best possible upper bound is simply the general bound $4n-8$, which is attained for $n=8$ and for all $n\ge 10$.
\end{remark}

\subsection{Proof  overview and new tools}

Kaufmann et al. \cite{MR4821240} introduced the following density formula for a connected drawing $G$:
\[
e(G)= t(|V(G)|-2)
-\sum_{c\in \C}
\left(\frac{t-1}{4}\|c\|-t\right)
-|X|,
\]
where $t$ is a real number, $X$ is the set of crossings, $\C$ is the set of cells, and $\|c\|$ denotes the size of a cell $c$, that is, the number of vertices and edge-segments on its boundary, counted with multiplicity. This formula provides crucial support for the derivation and improvement of the edge number upper bounds of various beyond-planar graph classes.

Motivated by this viewpoint, 
we establish a new formula 
of $e(G)$
for a connected  $1$-plane graph $G$
(see Theorem~\ref{lem:maintool})
in terms of 
three planarization-based parameters $A, B$ and $C$
defined in Definition~\ref{def:pAB},
where $B$ and $C$ are 
always non-negative, 
and $A\ge 0$ as long as 
$G$ is $K_3$-free. 
This new formula therefore explicitly characterizes the correction terms relative to the planar bound $3n-6$, and serves as a crucial 
tool in the proofs of Theorems~\ref{thm:K3} and~\ref{thm:K4}.

In the $K_3$-free case, 
the new formula of $e(G)$
yields the upper 
bound $3n-6$ directly
(just as in the work of Bekos et al. using the  density formula). 
To show that this bound 
can be further improved to $3n-8$, 
we conduct a more refined structural analysis of \(K_3\)-free 1-plane graphs by introducing two core concepts: crossing skeleton and alternating $k$-vertices,
to disprove the existence 
of $K_3$-free $1$-planar 
graphs of order $n$ and 
size $3n-6$ or $3n-7$.

For the $K_4$-free case, we use the new formula of $e(G)$
to give a lower bound for the  number of crossings in a 
$1$-plane graph $G$, 
and then apply the fact that every $1$-planar drawing of an $n$-vertex graph has at most $n-2$ crossings
to deduce an upper bound of 
$e(G)$.
For the $K_5$-free case, we obtain the desired bound by starting from four suitable initial optimal $1$-plane graphs and
apply the $Q_4$-addition 
to define other desired 
optimal 
$K_5$-free graphs recursively.




\section{Preliminaries}\label{sec:notation}
\subsection{Terminology and notations}

Throughout this paper, all graphs are simple unless stated otherwise.
We assume that the reader is familiar with standard graph-theoretic terminology and notation; see \cite{MR4874150} for example.

A \emph{drawing} of a graph $G$ in the plane $\mathbb{R}^2$ represents each vertex by a distinct point and each edge by a simple curve joining its ends, with the interior of every edge disjoint from all vertices; an intersection point of the interiors of two edges is a \emph{crossing}. 
A drawing is \emph{good} if no edge crosses itself, no two adjacent edges cross each other, any two edges cross at most once, and no three edges cross at the same point.
A \emph{$1$-planar drawing} is a good drawing in which every edge is crossed at most once. 
A graph together with a fixed $1$-planar drawing is called a \emph{$1$-plane graph}, and a graph is \emph{$1$-planar} if it admits a $1$-planar drawing; throughout the paper,  all $1$-planar drawings are assumed to be good.

Let $G$ be a plane graph. The \emph{faces} of $G$ are the connected components of $\mathbb{R}^2\setminus G$.
For a face $F$ of $G$, the \emph{facial walk $\partial(F)$} of $F$ consists of one or more disjoint  closed walks obtained
by traversing the boundary of $F$; this facial walk may repeat vertices or edges.
We define $\deg(F)$ to be the length of this facial walk, i.e., the number of edge-occurrences on the boundary of $F$. 
  Let $\mathcal{F}(G)$ denote the set of faces of the plane graph $G$.
A face $F$ of $G$ is called a \emph{$k$-face} if $\deg(F)=k$ and a \emph{$k^+$-face} if $\deg(F)\ge k$.
Let $\mathcal{F}_k(G)$ and $\mathcal{F}_{\ge k}(G)$ denote the sets of all $k$-faces and $k^+$-faces of $G$, respectively.

Let $G$ be a $1$-plane graph. The \emph{planarization} of $G$, denoted by $G^\times$, is the plane graph obtained by replacing each crossing of two edges $uv$ and $xy$ with a new vertex $z$ of degree 4 and subdividing these edges at $z$, that is, replacing $uv$ and $xy$ by the four edges $uz$, $zv$, $xz$, and $zy$.
Clearly, $G^\times=G$ if $G$ is a plane graph.
Vertices of $G^\times$ that correspond to crossings in the fixed drawing of $G$ are called \emph{fake vertices}.
All other vertices of $G^\times$ (that is, the original vertices of $G$) are called \emph{true vertices}.
Let $V_{\mathrm{fake}}(G^\times)$ denote the set of fake vertices of $G^\times$, and put
$x:=|V_{\mathrm{fake}}(G^\times)|$ (i.e. the number of fake vertices). An edge of $G^\times$ whose endpoints are both true vertices is called a \emph{true edge}.
An edge of $G^\times$ that is incident with at least one fake vertex is called a \emph{fake edge}.
 Two faces incident with a fake vertex $z$ are said to be \emph{opposite around $z$ } if they share only the vertex $z$.
 A face $F$ of $G^\times$ is called a \emph{true face} if $\partial(F)$ contains only true vertices,
and a \emph{fake face} if $\partial(F)$ contains at least one fake vertex.


For a graph $G$ and a vertex $u\in V(G)$, let $N_G(u)$ denote the set of neighbors of $u$ in $G$, and let $d_G(u)=|N_G(u)|$ denote the degree of $u$ in $G$.
The \emph{minimum degree} of $G$ is
$
\delta(G):=\min\{d_G(v): v\in V(G)\}.
$
We write $n_i(G)$ for the number of vertices of degree $i$ in a graph $G$. 
Let \(S\) be a set of edges with both end-vertices in \(V(G)\). Let \(G+S\) denote the graph obtained from \(G\) by adding all edges in \(S\).


We use $K_k$, $P_k$ and $C_k$ to denote the complete graph, the path graph and the cycle graph on $k$ vertices, respectively.

\subsection{Some useful lemmas}

We write $T_{k-1}(n)$ for the Tur\'an graph: the complete $(k-1)$-partite graph on $n$ vertices
whose part sizes differ by at most one.
Equivalently, letting $q=\lfloor n/(k-1)\rfloor$ and $r:=n\bmod (k-1)$, $T_{k-1}(n)$ has
$r$ parts of size $q+1$ and $k-1-r$ parts of size $q$.
Below is the well-known Tur\'an’s Theorem, which will be applied to handle the cases of smaller order.
\begin{lem}[Tur\'an \cite{MR0018405}]

Let $k\ge 2$ and let $G$ be a $K_k$-free graph of order $n$, with $r:=n\bmod (k-1)$.
Then
$
e(G)\le \frac12\cdot \frac{(k-2)(n^2-r^2)}{k-1}+\binom{r}{2}$.
Moreover, equality holds if and only if $G$ is isomorphic to $T_{k-1}(n)$.
\end{lem}

Note that 
$e(T_{k-1}(n))
\le \frac{(k-2)n^2}{2(k-1)}$ 
also holds.

\begin{lem}[Bodendiek et al.\cite{MR0732806}, Suzuki \cite{MR2746706}]\label{lem:edges_1planar}
Let $G$ be a  $1$-planar graph of order $n$. Then

\[
e(G) \le
\begin{cases}
\binom{n}{2}, & 1 \le n \le 6; \\[6pt]
4n - 9, & n = 7,9; \\[6pt]
4n - 8, & \text{otherwise}.
\end{cases}
\]
\end{lem}

\begin{lem}[Brandenburg et al. \cite{MR3067240}, or  Zhang and Li \cite{MR4218488}] \label{lem:fakevertices}
Let $G$ be a $1$-plane graph, and let $G^{\times}$ be its planarization.  Then the following statements hold.
\begin{itemize}
\item[(i)] Fake vertices of $G^{\times}$  are pairwise non-adjacent;
\item[(ii)] for each face $F$ of $G^{\times}$, the number of occurrences of fake vertices on the facial walk of $F$
is at most $\bigl\lfloor\frac{\deg(F)}{2}\bigr\rfloor$.
\end{itemize}
\end{lem}%

\begin{lem}[Euler's formula, see page 101 in \cite{MR4874150}]\label{lem:euler}
Let $G$ be a connected plane graph with $n$ vertices, $m$ edges, and $f$ faces. Then
$
n-m+f=2.
$
\end{lem}

The following lemma is well-known.
\begin{lem}[see page 45 in \cite{MR0898434}]\label{lem:plane2-connected}
Let $G$ be a plane graph. If every face of $G$ is bounded by a cycle, then $G$ is $2$-connected, and hence $\delta(G)\ge 2$.
\end{lem}

\begin{obs}\label{obs_no1or2faces}
Let $G$ be a  connected 
plane graph of order  $n\ge 3$. Then every face of $G$ is a $3^+$-face.
\end{obs}

The next lemma follows immediately from Euler's formula and the handshake lemma on faces for plane graphs. 
The proof is left to readers.

\begin{lem}
\label{lem:planeedge_count}
Let $G$ be a connected plane graph of order $n\ge 3$. Then
\begin{equation*}\label{eq:planeeq}
e(G)= 3n - 6 - \sum_{F \in \mathcal{F}_{\ge 4}(G)} \bigl(\deg(F) - 3\bigr).
\end{equation*}
\end{lem}

The next lemma establishes a lower bound on 
$n_2(G)+n_3(G)$
in a connected plane graph $G$ that has only 4-faces apart from some exceptional faces,
and will be used 
in the proof of Theorem~\ref{thm:K3}.

\begin{lem}\label{lem:exceptional-faces-lowdegree}
Let \(G\) be a connected plane graph with \(\delta(G)\ge 2\).
If 
$F_1,\dots,F_t$ are the faces
of $G$ 
with $\deg(F_i)\ne 4$
and 
$s:=\sum_{i=1}^t \bigl(\deg(F_i)-4\bigr)$, 
then
\[
n_2(G)+n_3(G)\ge \left\lceil \frac{s+n_3(G)}{2}\right\rceil +4 .
\]

\end{lem}\begin{proof}
Let \(n,m,f\) be the numbers of vertices, edges, and faces of \(G\), respectively.
Since every face of \(G\) is a \(4\)-face except for exactly \(t\) faces
\(F_1,\dots,F_t\), we have
$$
2m =4(f-t)+\sum_{i=1}^t \bigl(\deg(F_i)\bigr)
=4(f-t)+\sum_{i=1}^t \bigl(\deg(F_i)-4\bigr)+4t
=4f+s.
$$
Together with Euler's formula \(n-m+f=2\), this gives
$
2m=4n-s-8.
$
Hence
\[
s+8=4n-2m=\sum_{v\in V(G)}(4-d_G(v))
\le 2n_2(G)+n_3(G),
\]
where the last inequality follows 
from the condition $\delta(G)\ge 2$.
Since \(n_2(G)+n_3(G)\) is an integer, it follows that
$$
n_2(G)+n_3(G)\ge 
\left\lceil \frac{s+n_3(G)}{2}\right\rceil +4.
$$ 
Therefore, we complete the proof.
\end{proof}

\begin{cor}
	\label{cor:lowdegree-special}
Let \(G\) be a connected plane graph with \(\delta(G)\ge 2\). Then the following statements hold:
\begin{enumerate}
	[label=(\roman*)]

\item if every face of \(G\) is a  \(4\)-face except for one \(k\)-face,
then
$n_2(G)+n_3(G)\ge 
\left\lceil \frac{k}{2}\right\rceil +2$, and 

\item if every face of \(G\) is a \(4\)-face except for exactly two \(6\)-faces, then
$
n_2(G)+n_3(G)\ge 6.
$
Moreover, if \(n_3(G)\ge 1\), then
$
n_2(G)+n_3(G)\ge 7.
$
\end{enumerate}
\end{cor}

Corollary~\ref{cor:lowdegree-special} follows directly from 
Lemma~\ref{lem:exceptional-faces-lowdegree}.
Its proof is left to readers.

\section{An edge-count formula of $1$-planar graphs }


In this section, we establish a specialized edge-count formula for 1-planar graphs. Before doing so, we first introduce some related concepts and notation.

Let \(G\) be a plane graph. For a vertex \(u\in V(G)\) and a face \(F\) of \(G\), define the \emph{incidence multiplicity} of \(u\) with \(F\), denoted by $\eta(u,F)$, to be the number of occurrences of \(u\) on the facial walk of \(F\). 
Thus we say that \(u\) is incident with \(F\) exactly \(s\) times if $\eta(u,F)=s$. 
For instance, if $G$ is the graph shown in Figure~\ref{fig:face_incidence},
then 
$$
\eta(u,F^0)=0,\quad 
\eta(u,F^1)=2, \quad 
\eta(u,F^2)=\eta(u,F^3)=1.
$$

\begin{figure}[H]
\centering
\begin{tikzpicture}[scale=0.6]
	\begin{pgfonlayer}{nodelayer}
		\node [style=whitenode] (0) at (-2, 2) {}; 
		\node [style=whitenode] (1) at (2, 2) {};
		\node [style=whitenode] (2) at (-1, -1){};
		\node [style=whitenode] (3) at (2, -2){}; 
		\node [style=whitenode] (4) at (-2, -2){}; 
		\node [style=whitenode] (5) at (0, 0) {$u$};
		\node [style=none] (6) at (0, -1.5){$F^1$};
		\node [style=whitenode] (8) at (3, 0) {};
		\node [style=none] (9) at (1.5, 0) {$F^2$};
		\node [style=none] (10) at (0, 1.25) {$F^3$};
		\node [style=none] (11) at (-3.25, 0) {$F^0$};
	\end{pgfonlayer}
	\begin{pgfonlayer}{edgelayer}
		\draw (0) to (3);
		\draw (1) to (2);
		\draw (0) to (1);
		\draw (0) to (4);
		\draw [bend left=360, looseness=0.75] (3) to (4);
		\draw (8) to (3);
		\draw (1) to (8);
	\end{pgfonlayer}
\end{tikzpicture}
\caption{Incidence multiplicities of the vertex $u$ with faces $F^0,F^1,F^2,F^3$}
\label{fig:face_incidence}
\end{figure}
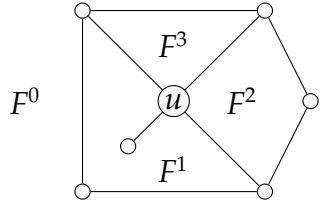

For any subset \(\Gamma\) of 
$\F(G)$, 
let
$\eta(u,\Gamma):=\sum\limits_{F\in \Gamma} \eta(u,F)$.
That is, \(\eta(u,\Gamma)\) is the sum of the incidence multiplicities of \(u\) with the faces in \(\Gamma\).

\begin{obs}\label{ob:vertex_mul_number}
	Let \(G\) be a plane graph and $u$ be a vertex of $G$. 
	Then
	$\eta(u,\mathcal{F}(G))=d_G(u) $, and  $0\le \eta(u,F)\le d_G(u) $ for each face $F\in  \mathcal{F}(G)$. 
\end{obs}

Let  \(G\) be a 1-plane graph
with the planarization $G^\times$. For \(z\in V_{\mathrm{fake}}(G^\times)\), we define
\equ{eq3-1}
{
a(z):=\eta(z,\mathcal F_{\ge 4}(G^\times)).
}
Note that $0\le a(z)=4-\eta(z,\mathcal F_{3}(G^\times))\le 4$.
For any $F\in\mathcal F(G^\times)$, we define 
\equ{eq3-2}
{
c(F):=\sum_{u \in \partial(F)\cap 
V_{\mathrm{fake}}(G^\times)} \eta(u,F).
}
%

\begin{lem}
\label{lem:eqof_incidences_faces_fakev}
	Let $G$ be a $1$-plane graph with the planarization $G^\times$. Then
	\[
	\sum_{z\in V_{\mathrm{fake}}(G^\times)} a(z)=\sum_{F\in\mathcal F_{\ge4}(G^\times)} c(F).
	\]
\end{lem}
\begin{proof}
	Let $\Phi$ be the multiset of incidences between fake vertices of $G^\times$ and faces in $\mathcal F_{\ge4}(G^\times)$, where each incidence $(z,F)$ appears with multiplicity equal to the incidence multiplicity of $z$ with $F$. Thus, we have

	$$
	\sum_{z\in V_{\mathrm{fake}}(G^\times)} a(z)=|\Phi|= \sum_{F\in\mathcal F_{\ge4}(G^\times)} c(F).
	$$
\end{proof}

\begin{defi} 
	\label{def:pAB}
For any $1$-plane graph $G$,
we define invariants $A$, $B$ and $C$ as follows:
\begin{align}
A(G)&:= \sum_{z\in V_{\mathrm{fake}}(G^\times)}\bigl(a(z)-2\bigr), 
\label{eq:def-A}\\
B(G)&:= \sum_{F\in\mathcal F_{\ge 4}(G^\times)}
\Bigl(\tfrac{\deg(F)}{2}-c(F)\Bigr),
\label{eq:def-B}\\
C(G)&:= \sum_{F\in\mathcal F_{\ge 4}(G^\times)}\bigl(\deg(F)-4\bigr). 
\label{eq:def-p}
\end{align}
\end{defi}

In contexts where no confusion arises, $A(G)$, $B(G)$, and $C(G)$ will be denoted simply by $A$, $B$, and $C$, respectively.
A corollary follows directly from 
Lemma~\ref{lem:eqof_incidences_faces_fakev}.

\begin{cor}\label{cor3-1}
	For any $1$-planar graph $G$, 
$$
A+B=\frac 12 
\sum_{F\in\mathcal F_{\ge 4}(G^\times)} \deg(F)
-2|V_{fake}(G^\times)|.
$$
\end{cor}

In this section, we will present some properties on $A(G), B(G)$ and $C(G)$, which will be applied in  proving our main results.

\subsection{Edge-count formula }

The following theorem provides a convenient edge-count identity for a connected $1$-plane graph in terms of the above parameters $A$, $B$ and $C$, and it will serve as a  tool in the proofs of Theorems~\ref{thm:K3} and~\ref{thm:K4}.

\begin{thm}[\textbf{Edge-count formula}]\label{lem:maintool}
Let $G$ be a connected $1$-plane graph with $n\ge 3$ vertices and $m$ edges.
Then, 
\begin{equation*}
m = 3n-6  - \frac{A+B}{2}-\frac34\,C.
\end{equation*}
\end{thm}

\begin{proof}
Let $n':=|V(G^\times)|$,  $m':=|E(G^\times)|$
and 
 $x:=|V_{\mathrm{fake}}(G^\times)|$.
 It can be verified that
  $n'=n+x$ and $m'=m+2x$.
 
By Lemma \ref{lem:planeedge_count}, we have
$
m' =3n' - 6 - \sum_{F \in \mathcal{F}_{\ge 4}(G^\times)} \bigl(\deg(F) - 3\bigr).
$
As $n'=n+x$ and $m'=m+2x$, we have 
\Eqn{eq:edge_count_with_d_x}
{
m+2x&=&3(n+x)-6 -\sum_{F \in \mathcal{F}_{\ge 4}(G^\times)}
\bigl(\deg(F) - 3\bigr);
\nonumber  \\
m &=& 3n-6 -\sum_{F \in \mathcal{F}_{\ge 4}(G^\times)} \bigl(\deg(F) - 3\bigr)+x.
}

By Corollary~\ref{cor3-1} and 
the definition of $C$ in 
(\ref{eq:def-p}), 
\Eqn{eq:ABC}
{
\frac{A+B}{2}+\frac{3}{4}
C
&=&\frac14\sum_{F\in\mathcal F_{\ge4}(G^\times)}\deg(F)-x+\frac34\sum_{F\in\mathcal F_{\ge4}(G^\times)}\bigl(\deg(F)-4\bigr)\nonumber \\
&=&\sum_{F\in\mathcal F_{\ge4}(G^\times)}\frac14\deg(F)+\sum_{F\in\mathcal F_{\ge4}(G^\times)} \frac34\bigl(\deg(F)-4\bigr)-x
\nonumber\\
&=&\sum_{F\in\mathcal F_{\ge4}(G^\times)}\bigl(\deg(F)-3\bigr)-x.
}
Combining  (\ref{eq:edge_count_with_d_x}) with (\ref{eq:ABC}) yields that 
\[
m=3n-6-\frac{A+B}{2}
-\frac{3}{4}C,
\]
which completes the proof.
\end{proof}

\subsection{Properties of the 
invariants 
$A, B$ and $C$
} 

\begin{lem}
	\label{obs:pB-nonneg}
	Let $G$ be a $1$-plane graph of order $n \ge 3$.
Then $B\ge 0$ and $C\ge 0$, and 
\begin{enumerate}
	\item[(i)]
 $B=0$ holds if and only if $\frac{\deg(F)}{2}=c(F)$ for every $4^+$-face $F$ of $G^\times$ ; and 
 \item[(ii)]
 $C=0$ holds if and only if every $4^+$-face $F$ of $G^\times$  is a 4-face.
 \end{enumerate} 
\end{lem}

\begin{proof}
Recall that
$
B=\sum_{F\in\mathcal F_{\ge 4}(G^\times)}\Bigl(\frac{\deg(F)}{2}-c(F)\Bigr).
$
By Lemma~\ref{lem:fakevertices}(ii), $
\frac{\deg(F)}{2}-c(F)\ge 0
$ for every face $F\in\mathcal F_{\ge4}(G^\times)$, 
and thus $B\ge 0$.  Furthermore, $B=0$ holds if and only if 
$\frac{\deg(F)}{2}-c(F)=0$
(i.e.,  $\deg(F)/2=c(F)$)
for each $F\in\mathcal F_{\ge4}(G^\times)$.

By definition,
$
C=\sum_{F\in\mathcal F_{\ge4}(G^\times)}\bigl(\deg(F)-4\bigr).
$
Since $\deg(F)\ge 4$ for every $F\in\mathcal F_{\ge4}(G^\times)$, we have $C\ge0$.
 Furthermore, $C=0$ if and only if $\deg(F)-4=0$ 
 for each $F\in \F_{\ge4}(G^\times)$, 
 that is, every $4^+$-face $F$ is a 4-face.
\end{proof}


\begin{lem}[Bekos et al. \cite{Bekos2025}]\label{lem:fake-no-adjacent-triangles}
Let $G\in \Maxe{n}{3}$
and let 
$z$ be a fake vertex in $G^\times$.
Then there are at most two $3$-faces in $G^\times$ 
incident with $z$, 
and if two occur then the two $3$-faces  are opposite around $z$, as shown in Figure \ref{fig:fakevertexfaces}.
\end{lem}

\tikzstyle{dashpure1}=[-, dashed, draw={rgb,255: red,128; green,0; blue,128}]
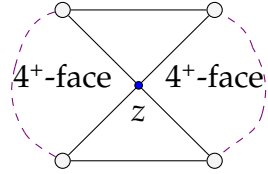
\begin{figure}[H]
\centering
\begin{tikzpicture}[scale=0.5]
	\begin{pgfonlayer}{nodelayer}
		\node [style=whitenode] (0) at (-2, 2) {};
		\node [style=whitenode] (1) at (2, 2) {};
		\node [style=whitenode] (2) at (-2, -2) {};
		\node [style=whitenode] (3) at (2, -2) {};
		\node [style=none] (4) at (0, -0.75) {$z$};
		\node [style=none] (5) at (-3.25, 0.5) {};
		\node [style=none] (6) at (-3, -1.25) {};
		\node [style=none] (7) at (3, 1) {};
		\node [style=none] (8) at (3.25, -0.75) {};
		\node [style=none] (9) at (-2, 0.25) {$4^+$-face};
		\node [style=none] (10) at (2, 0.25) {$4^+$-face};
		\node [style=bluenode] (11) at (0, 0) {};
	\end{pgfonlayer}
	\begin{pgfonlayer}{edgelayer}
		\draw (0) to (3);
		\draw (1) to (2);
		\draw (0) to (1);
		\draw (2) to (3);
		\draw [style=dashpure1, in=75, out=-165] (0) to (5.center);
		\draw [style=dashpure1, in=120, out=-105] (5.center) to (6.center);
		\draw [style=dashpure1, in=150, out=-60, looseness=0.75] (6.center) to (2);
		\draw [style=dashpure1, in=45, out=-105] (8.center) to (3);
		\draw [style=dashpure1, in=-60, out=75] (8.center) to (7.center);
		\draw [style=dashpure1, in=330, out=105] (7.center) to (1);
	\end{pgfonlayer}
\end{tikzpicture}

\caption{A fake vertex $z$ incident with exactly two opposite  $3$-faces.}
\label{fig:fakevertexfaces}
\end{figure}

\begin{lem}\label{obs:A-nonneg}
Let $G\in \Maxe{n}{3}$,
where $n\ge 4$.
Then $A\ge 0$,
where equality holds if and only if every fake vertex $z$ in $G^\times$
is incident with exactly  two $3$-faces,  
and these two $3$-faces are  opposite, shown in Figure \ref{fig:fakevertexfaces}.
\end{lem}

\begin{proof}
Fix a fake vertex $z$. By Lemma~\ref{lem:fake-no-adjacent-triangles}, $\eta(z, \mathcal F_3(G^\times))\le 2$, and thus by Observation \ref{ob:vertex_mul_number}, $a(z)\ge 2$. By definition, $A\ge 0$.
Moreover, $A=0$ holds if and only if $a(z)=2$ for every fake vertex $z$. Then $\eta(z, \mathcal F_3(G^\times))=2$. Furthermore, it is clear that any fake vertex is incident with at most two 3-faces, thus any fake vertex is incident with exactly  two $3$-faces. 
By Lemma~\ref{lem:fake-no-adjacent-triangles}, these two $3$-faces are opposite.
\end{proof}

By Lemmas~\ref{obs:pB-nonneg} 
and~\ref{obs:A-nonneg}, 
Bekos et al. \cite{Bekos2025}'s 
conclusion follows directly.

\begin{cor}\label{cor-3-3}
$\maxe{n}{3}\le 3n-6$ for $n\ge 4$.
\end{cor}

\begin{lem}\label{obs:A-nonneg_K4}
For any $G\in \Maxe{n}{4}$,
where $n\ge 3$, 
$A\ge - |V_{\mathrm{fake}}(G^\times)|$.

\end{lem}

\begin{proof} 
We first prove that 
 $a(z)\ge 1$ for every fake vertex $z$ of $G^\times$.
 
 Clearly, $a(z)\ge 0$.
 If \(a(z)=0\), then all faces incident with \(z\) are triangles.  
 We may assume $z$ is the crossing of $uv$ and $xy$ in $G$. The four triangular faces around \(z\)  force the cyclically consecutive pairs among \(u,x,v,y\) to be adjacent, and together with the crossing edges \(uv\) and \(xy\), this shows that \(u,x,v,y\) induce a \(K_4\), a contradiction.

Thus, $a(z)-2\ge -1$ for each fake vertex $z$.
By definition, 
$
A=\sum_{z\in V_{\mathrm{fake}}(G^\times)}
(a(z)-2)\ge -
|V_{\mathrm{fake}}(G^\times)|.
$
\end{proof}

\section{Preparation for proving Theorem \ref{thm:K3}}

Throughout this section, 
we assume that $G$ is a connected 1-plane (simple) graph of order $n(G)\ge 4$. 
Let $G^\times$ be the planarization  of $G$. Clearly, $G^\times$ is a connected plane graph of order 
$n(G^\times)\ge 4$.

\subsection{Types of small faces in $G^\times$}

A $k$-face $f$ in $G^\times$ is called \emph{alternating}
if true and fake vertices alternate in order along its boundary walk.
Now let 
$F$ be a $4$-face of $G^\times$. 
By Lemma~\ref{lem:fakevertices}(ii), we have
$
c(F)\le 2,
$
and hence $c(F)\in\{0,1,2\}$.
Moreover, by Lemma~\ref{lem:fakevertices}(i), no two fake vertices of $G^\times$ are adjacent, so
if $c(F)=2$ then the two fake vertices cannot be consecutive on the boundary walk of $F$.
Therefore they must be opposite on the boundary.
Accordingly, we classify all 4-faces of $G^\times$ into the following three types (see Figure~\ref{fig:Alpossible4-face}):
\begin{itemize}
\item[(i)] a \emph{true $4$-face} if $c(F)=0$, i.e., the boundary of $F$ consists of four true vertices.
\item[(ii)] a \emph{single-fake $4$-face} if $c(F)=1$, i.e., the boundary of $F$ contains exactly one fake vertex and three true vertices.
\item[(iii)] an \emph{alternating $4$-face} if $c(F)=2$; i.e., the boundary of $F$ alternates between true and fake vertices.
\end{itemize}

\begin{remark}
	\label{re:notrue}
The boundary  of an alternating $4$-face contains no true edge.
\end{remark}

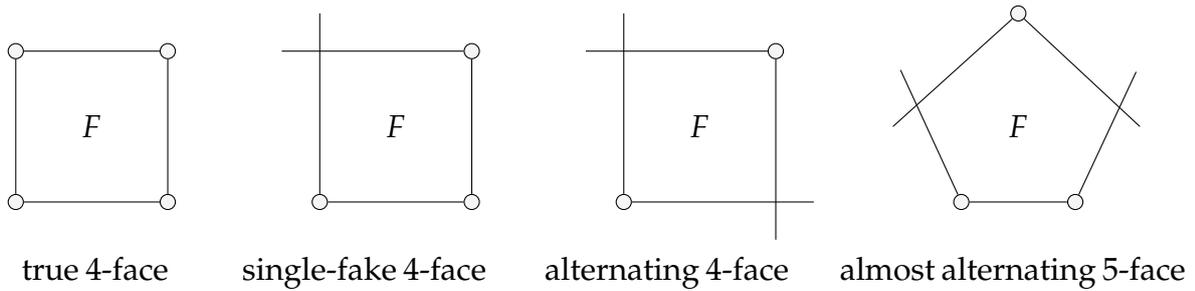
\begin{figure}[H]
	

\begin{minipage}[t]{0.64\textwidth}
\begin{tikzpicture}
	\begin{pgfonlayer}{nodelayer}
		\node [style=whitenode] (0) at (-6, 2) {};
		\node [style=whitenode] (2) at (-6, 0) {};
		\node [style=whitenode] (1) at (-4, 2) {};
		\node [style=whitenode] (3) at (-4, 0) {};
	\node [style=none] (23) at (-5, 1) {$F$};
		
	
		\node [style=whitenode] (4) at (-2, 0) {};
		\node [style=whitenode] (5) at (0, 2) {};
		\node [style=none] (6) 
		at (-2.5, 2) {};
		\node [style=none] (7) 
		at (-2, 2.5) {};
		\node [style=whitenode] (13) at (0, 0) {};
		\node [style=none] (24) at (-1, 1) {$F$};

		\node [style=whitenode] (14) at (2, 0) {};
		\node [style=whitenode] (15) at (4, 2) {};
		\node [style=none] (16) at (1.5, 2) {};
		\node [style=none] (17) at (2, 2.5) {};
		\node [style=none] (18) at (4, -0.5) {};
		\node [style=none] (19) at (4.5, 0) {};
		\node [style=none] (25) at (3, 1) {$F$};

	\end{pgfonlayer}
	\begin{pgfonlayer}{edgelayer}
		\draw (6.center) to (5);
		\draw (7.center) to (4);
		\draw (0) to (1);
		\draw (1) to (3);
		\draw (3) to (2);
		\draw (0) to (2);
		\draw (5) to (13);
		\draw (13) to (4);
		\draw (16.center) to (15);
		\draw (17.center) to (14);
		\draw (14) to (19.center);
		\draw (18.center) to (15);
	\end{pgfonlayer}
\end{tikzpicture}
\end{minipage}
\hspace{1.5 cm}
\begin{minipage}[t]{0.2\textwidth}
\begin{tikzpicture}
	\begin{pgfonlayer}{nodelayer}
		\node [style=whitenode] (0) at (0, 2.5) {};
		\node [style=none] (1) at (-1.55, 1.75) {};
		\node [style=none] (5) at (-1.65, 1) {};
		\node [style=whitenode] (2) at (-0.75, 0) {};
		\node [style=whitenode] (3) at (0.75, 0) {};
		\node [style=none] (6) at (1.55, 1.725) {};
		\node [style=none] (7) at (1.6, 1) {};
		\node [style=none] (8) at (0, 1) {$F$};
		\node [style=none] (9) at (0, -0.5) {};
	\end{pgfonlayer}
	\begin{pgfonlayer}{edgelayer}
		\draw (1.center) to (2);
		\draw (2) to (3);
		\draw (0) to (5.center);
		\draw (6.center) to (3);
		\draw (7.center) to (0);
	\end{pgfonlayer}
\end{tikzpicture}

\end{minipage}

{}
\hspace{0.1 cm}
true $4$-face 
\hspace{8 mm}
single-fake $4$-face 
\hspace{6 mm}
alternating $4$-face 
\hspace{5 mm}
almost alternating $5$-face

\caption{Possible types of 4-faces and almost alternating 
$5$-face}
\label{fig:Alpossible4-face}
\end{figure}

Now consider a $5$-face $F$.
By Lemma~\ref{lem:fakevertices},
the boundary of $F$ contains at most two fake vertices. 
If the boundary of $F$
contains exactly two fake vertices,
then $F$ must be as shown in
Figure~\ref{fig:Alpossible4-face},
and we call such a 5-face an \emph{almost alternating 5-face}.

The  \emph{crossing skeleton} of $G$, denoted by $\mathcal{C}(G)$,
is defined to be the spanning subgraph of $G$ whose 
edge set consists of all crossed edges of $G$.

\begin{obs}\label{obs:merge3-faces}
If a true edge $e=uv$ in $G^\times$ is incident with two fake $3$-faces, then in $\mathcal{C}(G)^\times$ these two fake 3-faces merge into an alternating $4$-face.
\end{obs}

\begin{lem}\label{lem:crossingskeleton_alternating}
 If every face of $G^\times$ is
 either a fake 3-face or an  alternating 4-face, 
 then  every  face of $\mathcal{C}(G)^\times$ is an alternating 4-face. 

\end{lem}

 \begin{proof}
Clearly, every alternating 4-face of $G^\times$ is still an alternating $4$-face of  $\mathcal{C}(G)^\times$.  
By the given condition, 
every true edge $uv$ is incident with fake faces. 
By Observation \ref{obs:merge3-faces}, every 4-face of \(\mathcal{C}(G)^\times\)  is alternating.
 \end{proof}

A vertex $v$ of $G^\times$ is called \emph{alternating} if every face of $G^\times$ incident with $v$ is an alternating $4$-face.
A vertex $v\in V(G)$ is called \emph{alternating} if it is alternating in $G^\times$.
In particular, an alternating  vertex of degree $k$ is called an \emph{alternating $k$-vertex}.

\begin{lem}
\label{lem:all-2fake-quadrangulation}
Any $1$-plane graph $G$ 
does not have any alternating $k$-vertex for $0\le k\le 2$.
\end{lem}

\begin{proof}
First, a true vertex of degree 0 or 1 in $G^\times$ cannot be incident with an alternating $4$-face $F$, because every vertex on a $4$-face is incident with two edges of the boundary of that face $F$. Suppose that there exists a vertex $u$ of degree $2$ in $G$.
Then $u$ is incident with exactly two $4$-faces; in particular, there is a $4$-face whose boundary walk is
$u\,x\,v\,y u$ where $x$ and $y$ are fake vertices and $v$ is a true vertex. The other $4$-face incident with $u$ must be bounded by the edge $ux$, $xa$, $by$, and $yu$ where $a,b $ are true vertices (see Figure \ref{fig:confdegree2}), which forces $a=b$,
implying that 
edges $va$ and  $vb$ in $G$ 
are parallel.
Thus, it contradicts the assumption that $G$ is simple.
\end{proof}

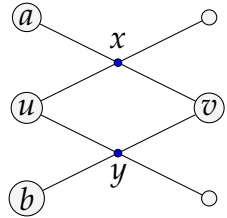
\begin{figure}[H]
\centering
\begin{tikzpicture}[scale=0.6]
	\begin{pgfonlayer}{nodelayer}
		\node [style=whitenode] (0) at (-2, 0) {$u$};
		\node [style=whitenode] (1) at (2, 0) {$v$};
		\node [style=whitenode] (2) at (-2, 2) {$a$};
		\node [style=whitenode] (3) at (2, 2) {};
		\node [style=whitenode] (4) at (-2, -2) {$b$};
		\node [style=whitenode] (5) at (2, -2) {};
		\node [style=none] (6) at (0, 1.5) {$x$};
		\node [style=none] (7) at (0, -1.5) {$y$};
		\node [style=bluenode] (8) at (0, 1) {};
		\node [style=bluenode] (9) at (0, -1) {};
	\end{pgfonlayer}
	\begin{pgfonlayer}{edgelayer}
		\draw (0) to (3);
		\draw (0) to (5);
		\draw (2) to (1);
		\draw (1) to (4);
	\end{pgfonlayer}
\end{tikzpicture}
\caption{Local configuration of $u$ of degree two where $x$ and $y$ are fake vertices}
\label{fig:confdegree2}
\end{figure}

\begin{lem}
\label{lem:degree3po}
For any  $G\in \Maxe{n}{3}$, 
if $u$ is an alternating vertex
of $G$,
then $d_G(u)\ge 4$.

\end{lem}

\begin{proof}
Suppose that $u$ is an
alternating vertex in a graph 
$G$ of $\Maxe{n}{3}$.
By Lemma~\ref
{lem:all-2fake-quadrangulation},
$d_G(u)\ge 3$.

Suppose that 
$N_G(u)=\{x,y,z\}$.
As $u$ is an alternating $3$-vertex, 
the local structure around $u$ is as shown in Figure~\ref{fig:structureofdegree3} (a),
where $F_1,F_2,$ and $F_3$ are the faces incident with $u$
each of which 
is an alternating $4$-face. 

Since $F_i$ is an alternating 4-face for $1\le i\le 3$, 
we have $u_1=u_6$, $u_2=u_3$, $u_4=u_5$,
implying that 
the subgraph of $G$ induced by $\{u_1,u_2,u_4\}$
is isomorphic to $K_3$,
as shown in Figure~\ref{fig:structureofdegree3} (b).
It contradicts the given condition that $G\in \Maxe{n}{3}$.
Hence the lemma holds. 
\end{proof}

\begin{figure}[H]
\centering
\begin{tikzpicture}[scale=0.7]
	\begin{pgfonlayer}{nodelayer}
		\node [style=whitenode] (0) at (-2.25, -0.25) {$u$};
		\node [style=whitenode] (1) at (-2.25, 1.75) {$x$};
		\node [style=whitenode] (2) at (-0.5, -2) {$y$};
		\node [style=whitenode] (6) at (-4, -2) {$z$};
		\node [style=whitenode] (7) at (-2.75, 0.75) {$u_1$};
		\node [style=whitenode] (8) at (-1.75, 0.75) {$u_2$};
		\node [style=whitenode] (9) at (-3.75, -1) {$u_6$};
		\node [style=whitenode] (10) at (-3, -1.75) {$u_5$};
		\node [style=whitenode] (11) at (-1.75, -1.75) {$u_4$};
		\node [style=whitenode] (12) at (-0.75, -1) {$u_3$};
		\node [style=none] (13) at (-1.25, 0) {$F_1$};
		\node [style=none] (14) at (-2.25, -1) {$F_2$};
		\node [style=none] (15) at (-3.25, 0) {$F_3$};
		
		
		\node [style=whitenode] (16) at (4.25, -0.25) {$u$};
		\node [style=whitenode] (17) at (4.25, 1.75) {$x$};
		\node [style=whitenode] (18) at (6, -2) {$y$};
		\node [style=whitenode] (19) at (2.5, -2) {$z$};
		\node [style=whitenode] (22) at (3.5, 0.75) {$u_1$};
		\node [style=whitenode] (24) at (4.25, -1.75) {$u_4$};
		\node [style=whitenode] (25) at (5, 0.75) {$u_2$};
	\end{pgfonlayer}
	\begin{pgfonlayer}{edgelayer}
		\draw (0) to (1);
		\draw (0) to (2);
		\draw (0) to (6);
		\draw (7) to (8);
		\draw (12) to (11);
		\draw (10) to (9);
		\draw (16) to (17);
		\draw (16) to (18);
		\draw (16) to (19);
		\draw (25) to (24);
		\draw (24) to (22);
		\draw (22) to (25);
	\end{pgfonlayer}
\end{tikzpicture}

(a) \hspace{4 cm} (b)

\caption{The local  structure of an alternating $3$-vertex $u$}
\label{fig:structureofdegree3}
\end{figure}
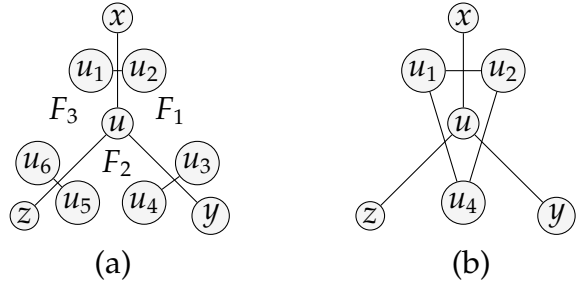

\subsection{Faces in $G^\times$
	for $G\in \Maxe{n}{3}$
}

We first prove 
an important 
fact that
 for any $G\in \Maxe{n}{3}$, 
if each face in $G^\times$ is bounded by a cycle, 
then
there are at least two faces
 in $G^\times$  
which are neither alternating $4$-faces nor fake $3$-faces.

\begin{lem}
	\label{lem:3n-7-0}
	For any $G\in \Maxe{n}{3}$, 
	$G^\times$ does not 
	contain a face $F$ which 
	is bounded by a cycle
	such that each face in $\F(G^\times)\setminus \{F\}$
	is either an alternating $4$-face or a fake $3$-face. 
	\end{lem}

\begin{proof}
Suppose that $G$ 
is a graph in $\Maxe{n}{3}$
and 
$G^\times$  contains a
face $F$ 
which 
is bounded by a cycle
such that each face in $\F(G^\times)\setminus \{F\}$
is either an alternating $4$-face or a fake $3$-face.

Let $H=\C(G)$. 
Then,  $H^\times$ is obtained 
from $G^\times$ by removing 
all true edges in $G^\times$.
Thus $F$ is expanded into 
an alternating $2r$-face 
$F^*$ in $H^\times$
for some $r\ge 2$ after
the removal of all true edges on $\partial(F)$.

By Remark~\ref{re:notrue},
every true edge in $G^\times$
cannot be on the boundary of an alternating $4$-face. 
Since each face in $\F(G^\times)\setminus \{F\}$
is either an alternating $4$-face or a fake $3$-face, 
every true edge of $G^\times$
 which is not on $\partial(F)$ 
is incident with exactly two fake $3$-faces in $G^\times$,
and thus such two fake 
$3$-faces are  merged into an  alternating $4$-face in $H^\times$. 
Hence 
all faces in $\F(H^\times)\setminus \{F^*\}$
are  alternating 4-faces.

Since each face in $H^\times$
is bounded by a cycle, 
by Lemma \ref{lem:plane2-connected}, 
$H^\times$ is $2$-connected
and thus
$\delta(H^\times)\ge 2$.
By Corollary~\ref{cor:lowdegree-special} (i),
$n_2(H^\times) +n_3(H^\times)
\ge r+2$.
Note that there are only $r$ true vertices on $\partial(F^*)$
and each fake vertex in $H^\times$ is of degree $4$.
Thus, 
$n_2(H^\times) +n_3(H^\times)
\ge r+2$ implies that 
$H^\times$ has a vertex $p$
which is 
not on $\partial(F^*)$ with $d_{H^\times}(p)=2$ or $3$.
Clearly, $p$ is a true vertex in 
$H^\times$.

Since 
$p$ is a true vertex in $H^\times$
and $p$ is not on $\partial(F^*)$, 
$p$ must be incident 
with $4$-faces of $H^\times$ only,
implying that $p$ is 
an alternating vertex in $H^\times$.
Since $G\in \Maxe{n}{3}$,
we have $H\in \Maxe{n}{3}$.
By Lemma \ref{lem:degree3po},
$d_{H^\times}(p)=d_H(p)\ge 4$, 
contradicting the fact that 
$d_{H^\times}(p)\le 3$.
Hence the lemma holds.
\end{proof}

\begin{lem}
	\label{lem:3n-7-2}
Let $G\in \Maxe{n}{3}$
		with $e(G)=3n-7$,
		where $n\ge 4$.
Then, $G^\times$ has 
two single-fake
$4$-faces $F_1$
and $F_2$
such that 
all other $4^+$-faces of 
$G^\times$ are alternating
$4$-faces. 
Moreover, every fake vertex 
in $G^\times$
is incident with exactly two opposite fake 3-faces.

\end{lem}

\begin{proof}
	As $e(G)=3n-7$,
by Theorem~\ref{lem:maintool}, we have
$1=\frac{A+B}{2}+\frac34 \, C.
$
Note that $A\ge0$ by Lemma~\ref{obs:A-nonneg},
and $B\ge0$, $C\ge0$ by Lemma~\ref{obs:pB-nonneg}.
Clearly, $A$ and $C$ are integers.   
Thus $C=0$ or $C=1$.

\begin{claim}\label{cl:C=0}
$C=0$
\end{claim}

\begin{proof} 
Suppose that $C=1$.
Then, 
$1=\frac{A+B}{2}+\frac34 C$
implies that $A=0$ and $B=\frac{1}{2}$. 
By (\ref{eq:def-p}), 
$G^\times$ has exactly 
one $5$-face $F$, and all other $4^+$-faces of $G^\times$ are $4$-faces.
Furthermore, since $B=\frac{1}{2}$, 
(\ref{eq:def-B}) implies that
$F$ is almost alternating and 
all $4$-faces of $G^\times$ are alternating.
Since $G\in \Maxe{n}{3}$, 
every $3$-face  in $\F(G^\times)\setminus \{F\}$ is a fake $3$-face.
Thus, it leads to 
a contradiction with Lemma~\ref{lem:3n-7-0}.
Hence $C=0$.
\end{proof}

\begin{claim}\label{cl:A0B2}
	$(A,B)=(0,2)$.
\end{claim}

\begin{proof} 
	By Claim~\ref{cl:C=0}, 
we have $A+B=2$, 
implying that 
 $(A,B) \in \{(2,0), (1,1), (0,2)\}$.
As $C=0$, 	by (\ref{eq:def-p}),
$\deg(F)=4$ for all $F\in\mathcal F_{\ge4}(G^\times)$,
i.e., all $4^+$-faces of $G^\times$ are $4$-faces,
implying that 
every face of $G^\times$ is 
either a fake $3$-face or a
$4$-face.

Now we are going to show that 
 $(A,B)=(0,2)$.
	
If $(A,B)=(2,0)$,   then by 
	(\ref{eq:def-B}), 
$c(F)=2$
	for every $F\in\mathcal F_{\ge4}(G^\times)$.
By Lemma \ref{lem:crossingskeleton_alternating},  every 4-face of $H^\times$  is  an  alternating 4-face.

If $(A,B)=(1,1)$,  then 
by  (\ref{eq:def-B}), 
$c(F)=1$ for exactly one face $F\in\mathcal F_{4}(G^\times)$,
and $c(F')=2$
for all other faces in $\mathcal F_{4}(G^\times)$.
Thus, $F$ is the only $4$-face 
of $G^\times$
which is not alternating.

However, 
both cases above lead to 
a contradiction with Lemma~\ref{lem:3n-7-0}.
Hence $(A,B)=(0,2)$.
\end{proof}

By Claim~\ref{cl:A0B2}, 
$(A,B)=(0,2)$.
By  (\ref{eq:def-A}),
$a(z)=2$ for every fake vertex $z$, and thus Lemma \ref{lem:fake-no-adjacent-triangles} yields that 
$z$ is incident with exactly two opposite fake 3-faces.
As $B=2$, 
by  (\ref{eq:def-B}), either
\begin{itemize}
	\item [(i)] 
	$c(F)=0$ for exactly one face $F\in\mathcal F_{4}(G^\times)$,
	and $c(F')=2$
	for all faces $F'\in \mathcal F_{4}(G^\times)\setminus \{F\}$; or 

	\item [(ii)]  
	$c(F_1)=c(F_2)=1$ for exactly two faces $F_1,F_2\in\mathcal F_{4}(G^\times)$,
	and $c(F')=2$
	for all  $F'\in \mathcal F_{4}(G^\times)\setminus \{F_1,F_2\}$.
\end{itemize}
However, if case (i) happens, then each face in $\F(G^\times)\setminus \{F\}$ 
is an alternating $4$-face, 
contradicting 
Lemma~\ref{lem:3n-7-0}.
Thus, case (ii) occurs,
and the lemma holds.
	\end{proof}

\section{Proof of Theorem \ref{thm:K3}}

Corollary~\ref{cor-3-3} has shown that 
$\maxe{n}{3}\le 3n-6$.
In this section, we will 
prove Theorem~\ref{thm:K3}
by showing that 
$\maxe{n}{3}\ne 3n-6$
and $\maxe{n}{3}\ne 3n-7$.

\subsection{$\maxe{n}{3}\ne 3n-6$}

\begin{lem}
	\label{lem:con} 
	Let $G\in \Maxe{n}{3}$,
	where $n\ge 4$. 
	If $e(G)\ge 3n-8$, then $G$ is connected. 
\end{lem}

\begin{proof} 
Suppose that $G$ is disconnected, with components $G_1,\dots,G_t$ $(t\ge2)$. 
	Put $n_i:=|V(G_i)|$
	for $1\le i\le t$
	and $e_i:=e(G_i)$. 
	By the given condition, 
	for $1\le i\le t$, 
	$e_i\le n_i-1$ if $n_i\le 3$,
	and 
	by
	Corollary~\ref{cor-3-3}, 
	$e_i \le 3n_i - 6$ if $n_i\ge 4$.
	
	For $i=1,2,3$,
	let $t_i$ be the number of 
	$j$'s in $\{1,2,\cdots, t\}$ with
	$|n_j|=i$, and let $s=t-t_1-t_2$.
	Then
	\Eqnn{
		e(G) &\le &
		0+t_2+2t_3+
		\sum_{1\le i\le t\atop n_i\ge 4}(3n_i-6)\\
		&=&t_2+2t_3+3(n-t_1-2t_2-3t_3)
		-6(s-t_3)\\
		&=&3n-(6s+3t_1+5t_2+t_3)
		\le 3n-(6s+3t_1+3t_2).
	}
	Since $e(G)\ge 3n-8$, the above inequality yields that 
	$$
	8 
	\ge 6s+3(t_1+t_2)=6s+3(t-s)
	=3s+3t\ge 3s+6, 
	$$
	implying that $s=0$.
	However, 
	$s=0$ implies that
	each $G_j$ is either $K_1$ 
	or $K_2$, and thus $e(G) \le \lfloor n/2\rfloor\le 3n-9$ for $n\ge 4$, a contradiction.
	Hence $G$ is connected.
\end{proof}

\begin{lem}
	\label{lem:3n-6}

		$\maxe{n}{3}\ne 3n-6$ for $n\ge 4$.
\end{lem}

\begin{proof} 
	Suppose that
		$G\in \Maxe{n}{3}$ with
		$e(G)\ne 3n-6$,
		where  $n\ge4$.
		By Lemma \ref{lem:con}, $G$ is connected. 
		By Theorem~\ref{lem:maintool}, 
		$
		e(G)=3n-6-\frac{A+B}{2} -\frac34\,C.
		$
		Furthermore,   $A\ge0$ by Lemma~\ref{obs:A-nonneg}, and $B\ge0$, $C\ge0$ by Lemma~\ref{obs:pB-nonneg}.
		Hence $e(G)=3n-6$
		implies that $A=B=C=0$.
		
		As $B=C=0$,  Lemma~\ref{obs:pB-nonneg} 
		yields that every $4^+$-face of $G^\times$ is a $4$-face,
		and  for each $4$-face $F$ in $G^\times$, 
		$\deg(F)/2=c(F)$, 
		i.e., $F$ is alternating. 
		Thus, every face in $G^\times$ 
		is either a fake $3$-face
		or an alternating $4$-face,
		contradicting Lemma~\ref{lem:3n-7-0}.

		Hence the result holds. 
\end{proof}


\subsection{$\maxe{n}{3}\ne 3n-7$}

\begin{prop}
	\label{lem:3n-7}
	$\maxe{n}{3}\ne 3n-7$ for $n\ge 4$.
\end{prop}

\begin{proof} 
Suppose Proposition~\ref{lem:3n-7}
fails, and 	there exists 
$G\in \Maxe{n}{3}$, 
where $n\ge 4$,
with $e(G)=3n-7$.
We may assume that  
such a graph $G$ is chosen so that $n$ has the minimum value.
Clearly, $n\ge 5$,
as any graph in $\maxe{4}{3}$
has at most $4$ edges.
Then, $\delta(G)\ge 4$;
otherwise, for any $u\in V(G)$
with $d_G(u)\le 3$, 
$G-u\in \maxe{{n-1}}{3}$
and 
$e(G-u)\ge 3n-7-3=3(n-1)-7$,
a contradiction with the
minimality of $n$ or Lemma~\ref{lem:con}
or Lemma~\ref{lem:3n-6}.

By Lemma~\ref{lem:3n-7-2}, 
there are two single-fake faces $F_1$ and $F_2$ in $G^\times$
such that every  $4^+$-face 
in $\F(G^\times)\setminus \{F_1, F_2\}$ 
is an alternating $4$-face, 
and every fake vertex 
in $G^\times$
is incident with exactly two opposite fake 3-faces in $G^\times$.

Assume that \(F_1\) and \(F_2\) are the two single-fake 4-faces \(u v c w u\) and
\(u' v' c' w' u'\), respectively, shown in Figure~\ref{fig:caseiiib1a},
where $c$ and $c'$ are 
fake vertices in $G^\times$.

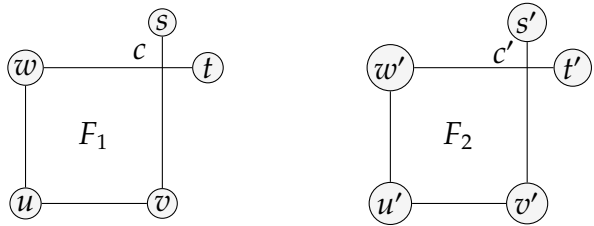
\begin{figure}[H]
\centering
\begin{tikzpicture}[scale=0.6]
	\begin{pgfonlayer}{nodelayer}
		\node [style=whitenode] (0) at (-5, 0) {$u$};
		\node [style=whitenode] (1) at (-5, 3) {$w$};
		\node [style=whitenode] (2) at (-2, 0) {$v$};
		\node [style=whitenode] (3) at (-2, 4) {$s$};
		\node [style=whitenode] (4) at (-1, 3) {$t$};
		\node [style=none] (9) at (-2.5, 3.35) {$c$};
		\node [style=none] (18) at (-3.5, 1.5) {$F_1$};
		
		
		\node [style=whitenode] (12) at (3, 0) {$u'$};
		\node [style=whitenode] (13) at (3, 3) {$w'$};
		\node [style=whitenode] (14) at (6, 0) {$v'$};
		\node [style=whitenode] (15) at (6, 4) {$s'$};
		\node [style=whitenode] (16) at (7, 3) {$t'$};
		\node [style=none] (17) at (5.5, 3.35) {$c'$};
		\node [style=none] (19) at (4.5, 1.5) {$F_2$};
	\end{pgfonlayer}
	\begin{pgfonlayer}{edgelayer}
		\draw (1) to (0);
		\draw (0) to (2);
		\draw (4) to (1);
		\draw (3) to (2);
		\draw (13) to (12);
		\draw (12) to (14);
		\draw (16) to (13);
		\draw (15) to (14);
	\end{pgfonlayer}
\end{tikzpicture}

\caption{$F_1$ and $F_2$}
\label{fig:caseiiib1a}

\end{figure}

We are now going to complete the proof by establishing the following claims.

\begin{claim}\label{cl:2edge}
$\partial(F_1)$ and $\partial(F_2)$ have at most 
one edge in common.
\end{claim}

\begin{proof}
Suppose the claim fails, 
and $\partial(F_1)$ and $\partial(F_2)$ have at  least
two edges in common.
Then, they share exactly two edges, as shown in Figure~\ref{share2edges} (a) or (b). 
However, observe from 
Figure~\ref{share2edges} (a) 
and (b) that 
$G$ contains some 
vertices of degree $2$
(see $u$ and $w$),
a contradiction to the conclusion 
that $\delta(G)\ge 4$.
Hence Claim~\ref{cl:2edge} holds.
\end{proof}

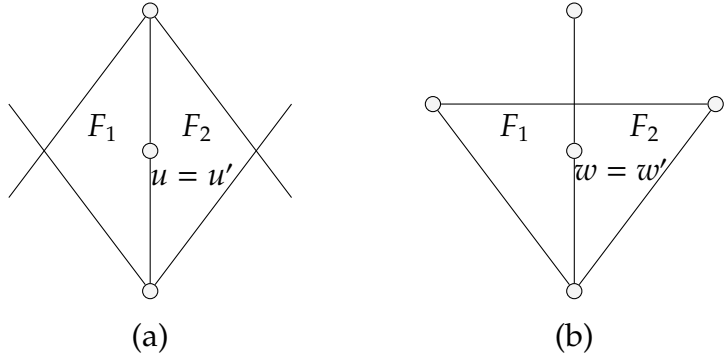
\begin{figure}[H]
	\centering
\begin{tikzpicture}[scale=0.62]
	\begin{pgfonlayer}{nodelayer}
		\node [style=whitenode] (0) at (-4, 16) {};
		\node [style=whitenode] (1) at (-4, 13) {};
		\node [style=whitenode] (2) at (-4, 10) {};
		\node [style=none] (3) at (-7, 14) {};
		\node [style=none] (4) at (-7, 12) {};
		\node [style=none] (5) at (-1, 14) {};
		\node [style=none] (6) at (-1, 12) {};
		\node [style=none] (7) at (-5, 13.5) {$F_1$};
		\node [style=none] (8) at (-3, 13.5) {$F_2$};
		\node [style=none] (9) at (-3.1, 12.5) {$u=u'$};
		\node [style=whitenode] (10) at (2, 14) {};
		\node [style=whitenode] (11) at (8, 14) {};
		\node [style=whitenode] (12) at (5, 10) {};
		\node [style=whitenode] (13) at (5, 13) {};
		\node [style=whitenode] (14) at (5, 16) {};
		\node [style=none] (15) at (3.75, 13.5) {$F_1$};
		\node [style=none] (16) at (6.5, 13.5) {$F_2$};
		\node [style=none] (17) at (6, 12.6) {$w=w'$};
		\node [style=none] (18) at (-4, 9) {(a)};
		\node [style=none] (19) at (5, 9) {(b)};
	\end{pgfonlayer}
	\begin{pgfonlayer}{edgelayer}
		\draw (0) to (4.center);
		\draw (3.center) to (2);
		\draw (2) to (5.center);
		\draw (0) to (6.center);
		\draw (0) to (1);
		\draw (1) to (2);
		\draw (10) to (12);
		\draw (12) to (11);
		\draw (11) to (10);
		\draw (13) to (12);
		\draw (14) to (13);
	\end{pgfonlayer}
\end{tikzpicture}

	
	
\caption{$F_1$ and $F_2$ 
have two edges in common 
on their boundary}

\label{share2edges}
\end{figure}

\begin{claim}\label{cl:1edge}
	$\partial(F_1)$ and $\partial(F_2)$ have no 
	edge in common.
\end{claim}

\begin{proof}
Suppose the claim fails.
By Claim~\ref{cl:2edge},
$\partial(F_1)$ and $\partial(F_2)$ have 
exactly one
edge in common.

We first show that 
$\partial(F_1)$ and $\partial(F_2)$ cannot share
 one true edge.
Otherwise, 
the structure formed by edges on $\partial(F_1)\cup  \partial(F_2)$ is 
 as shown in 
Figure~\ref{share1edge} (b) 
or (c), say (b).
Then, 
$\partial(F_1)$ 
and $\partial(F_2)$ 
share a true edge $uw$.
Let $G_0$ be the subgraph $G-uw$.
Observe that $G_0\in \Maxe{n}{3}$, 
and $F_1$ and $F_2$ are 
merged into one face $F$ in $G_0^\times$.
Furthermore, 
each face in $\F(G_0^\times)\setminus \{F\}$
is either an alternating $4$-face or a fake $3$-face,
a contradiction to Lemma~\ref{lem:3n-7-0}.

Thus, $\partial(F_1)$ and $\partial(F_2)$ share one 
fake edge and 
the structure formed by edges on $\partial(F_1)\cup  \partial(F_2)$ is as shown in 
Figure~\ref{share1edge} (a), where both $F_3$  and $F_4$ are fake $3$-faces since $a(z)=2$ and $F_1,F_2$ are 4-faces. But it is impossible because the two fake $3$-faces around $z$ must be opposite, whereas $F_3$ and $F_4$ are not, a contradiction.



Hence Claim~\ref{cl:1edge} holds.
\end{proof}

\begin{figure}[H]
	\centering
\begin{tikzpicture}[bezier bounding box, scale=0.4]
	\begin{pgfonlayer}{nodelayer}
		\node [style=whitenode] (0) at (-16, 2) {};
		\node [style=whitenode] (1) at (-8, 2) {};
		\node [style=whitenode] (2) at (-12, -2) {};
		\node [style=whitenode] (3) at (-16, -2) {};
		\node [style=whitenode] (4) at (-8, -2) {};
		\node [style=whitenode] (5) at (-12, 4) {};
		\node [style=none] (6) at (-10, 1.5) {$c=c'$};
		\node [style=none] (7) at (-14, 0) {$F_2$};
		\node [style=none] (8) at (-10, 0) {$F_1$};
		\node [style=none] (9) at (-12, -2.75) {$v=v'$};
		\node [style=none] (10) at (-16, -2.75) {$u'$};
		\node [style=none] (11) at (-16.8, 2) {$w'$};
		\node [style=none] (12) at (-7, 2) {$w$};
		\node [style=none] (13) at (-8, -2.75) {$u$};
		\node [style=none] (14) at (-12, 4.8) {$s=s'$};
		\node [style=none] (15) at (-13.5, 3) {$F_4$};
		\node [style=none] (16) at (-10.5, 3) {$F_3$};
		\node [style=none] (17) at (-12, -4) {(a)};
		\node [style=whitenode] (20) at (0, -2) {};
		\node [style=whitenode] (21) at (-4, -2) {};
		\node [style=whitenode] (22) at (4, -2) {};
		\node [style=none] (24) at (0.75, 2.8) {$w=w'$};
		\node [style=none] (25) at (-2, 0) {$F_2$};
		\node [style=none] (26) at (2, 0) {$F_1$};
		\node [style=none] (27) at (0, -2.75) {$u=u'$};
		\node [style=none] (28) at (-4, -2.825) {$v'$};
		\node [style=none] (30) at (5, 2) {};
		\node [style=none] (31) at (4, -2.75) {$v$};
		\node [style=none] (35) at (0, -4) {(b)};
		\node [style=none] (54) at (-4, 3) {};
		\node [style=none] (55) at (-5, 2) {};
		\node [style=whitenode] (56) at (0, 2) {};
		\node [style=none] (57) at (4, 3) {};
		\node [style=none] (58) at (-3.4, 2.475) {$c'$};
		\node [style=none] (59) at (4.35, 2.35) {$c$};
		\node [style=whitenode] (60) at (12, -2) {};
		\node [style=whitenode] (62) at (16, -2) {};
		\node [style=none] (63) at (12.75, 2.8) {$w=u'$};
		\node [style=none] (64) at (10, 0) {$F_2$};
		\node [style=none] (65) at (14, 0) {$F_1$};
		\node [style=none] (66) at (12, -2.75) {$u=v'$};
		\node [style=none] (67) at (8, -3) {};
		\node [style=none] (68) at (17, 2) {};
		\node [style=none] (69) at (16, -2.75) {$v$};
		\node [style=none] (70) at (12, -4) {(c)};
		\node [style=whitenode] (73) at (12, 2) {};
		\node [style=none] (74) at (16, 3) {};
		\node [style=none] (75) at (7.975, 2.775) {$w'$};
		\node [style=none] (76) at (16.35, 2.35) {$c$};
		\node [style=whitenode] (77) at (8, 2) {};
		\node [style=none] (78) at (7, -2) {};
		\node [style=none] (79) at (7.425, -1.525) {$c'$};
	\end{pgfonlayer}
	\begin{pgfonlayer}{edgelayer}
		\draw (0) to (1);
		\draw (1) to (4);
		\draw (4) to (2);
		\draw (2) to (3);
		\draw (3) to (0);
		\draw (5) to (2);
		\draw [ in=60, out=180] (5) to (0);
		\draw [ in=120, out=0] (5) to (1);
		\draw (22) to (20);
		\draw (20) to (21);
		\draw (54.center) to (21);
		\draw (55.center) to (56);
		\draw (56) to (30.center);
		\draw (57.center) to (22);
		\draw (56) to (20);
		\draw (62) to (60);
		\draw (73) to (68.center);
		\draw (74.center) to (62);
		\draw (73) to (60);
		\draw (77) to (73);
		\draw (77) to (67.center);
		\draw (60) to (78.center);
	\end{pgfonlayer}
\end{tikzpicture}

%
	
	\caption{$F_1$ and $F_2$ 
		have exactly one edge in common 
		on their boundary}
	
	\label{share1edge}
\end{figure}
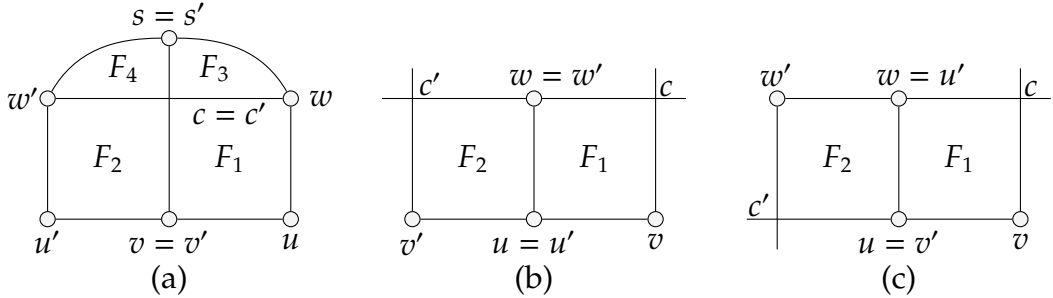

By Claim~\ref{cl:1edge}, 
$\partial(F_1)$ and $\partial(F_2)$ have 
no edge in common.

Since $a(c)=2$  and  \(u v c w u\) is a 4-face,  by Lemma \ref{lem:fake-no-adjacent-triangles}, $ws$ and $vt$ are  true edges in $G^\times$, and both $wscw$ and $vtcv$  bound fake 3-faces. 
Since the boundaries of faces $F_1$ and $F_2$ share no common edge, $uw$ is incident with a fake 3-face $wuc_2w$, and $uv$ is incident with a fake 3-face $uvc_1u$. Consequently, the structure near $F_1$ is depicted in Figure~\ref{fig:caseiiib1}.

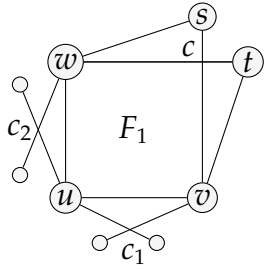
\begin{figure}[H]
\centering
\begin{tikzpicture}[scale=0.6]

	\begin{pgfonlayer}{nodelayer}

		\node [style=whitenode]
		 (20) at (-6, -1) {$u$};
		\node [style=whitenode] (21) at (-6, 2) {$w$};
		\node [style=whitenode] (22) at (-3, -1) {$v$};
		\node [style=whitenode] (23) at (-3, 3) {$s$};
		\node [style=whitenode] (24) at (-2, 2) {$t$};
		\node [style=none] (25) at (-3.3, 2.3) {$c$};
		\node [style=none] (26) at (-4.5, 0.5) {$F_1$};
		\node [style=whitenode] (32) at (-5.25, -2) {};
		\node [style=whitenode] (33) at (-4, -2) {};
		\node [style=whitenode] (34) at (-7, 1.5) {};
		\node [style=whitenode] (35) at (-7, -0.5) {};
		\node [style=none] (37) at (-4.5, -2.25) {$c_1$};
		\node [style=none] (38) at (-7, 0.5) {$c_2$};
	\end{pgfonlayer}

	\begin{pgfonlayer}{edgelayer}
		\draw (21) to (20);
		\draw (20) to (22);
		\draw (24) to (21);
		\draw (23) to (22);
		\draw (23) to (21);
		\draw (24) to (22);
		\draw (21) to (24);
		\draw (22) to (32);
		\draw (20) to (33);
		\draw (21) to (35);
		\draw (34) to (20);
	\end{pgfonlayer}
\end{tikzpicture}
\caption{ The local structure around $F_1$}
\label{fig:caseiiib1}
\end{figure}

Let $H:=\C(G)$.

\begin{claim}\label{claim:wv}
	Either
$d_{H^\times}(w)\ge 3$
or
$d_{H^\times}(v)\ge 3$.
\end{claim}

\begin{proof}
Suppose that 
$d_{H^\times}(w)\le 2$
and
$d_{H^\times}(v)\le 2$.

Since $c,c_2\in N_{H^\times}(w)$
and $c,c_1\in N_{H^\times}(v)$, 
we have 
$N_{H^\times}(w)=\{c,c_2\}$
and $ N_{H^\times}(v)=\{c,c_1\}$.
It follows that 
both $ws$ and $vt$ are 
incident with one fake $3$-face 
only, 
implying that both of them 
are incident with $F_2$.
Thus, both $ws$ and $vt$ 
are the only true 
edges on $\partial(F_2)$.
However, they have no end 
in common,
contradicting the fact that 
the two true edges in $F_2$ 
have one end in common.

Hence the claim holds.
\end{proof}

Note that $H^\times$ is obtained 
from $G^\times$ by the removal of all true edges in  $G^\times$.
Thus, $F_1$ in $G^\times$ will be expanded 
into an alternating $6$-face 
$F^*_1$ 
in 
$H^\times$ after the two true edges $uw$ and $uv$ are removed.
Similarly, $F_2$ in $G^\times$ will be expanded 
into an alternating 
$6$-face 
$F^*_2$  in $H^\times$.
By Remark~\ref{re:notrue}, 
each true edge of  
$G^\times$ which is not 
on $\partial(F_1)\cup \partial(F_2)$ 
is not incident with any $4$-face
of $G^\times$, and thus this edge
is on the boundary of 
two fake $3$-faces of $G^\times$.
Such two fake $3$-faces of $G^\times$ will be 
merged into an alternating 
$4$-face in $H^\times$.
Hence each face in $\F(H^\times)
\setminus \{F^*_1, F^*_2\}$
is an alternating $4$-face.


Since each face in $\F(H^\times)
\setminus \{F^*_1, F^*_2\}$
is an alternating $4$-face, 
by Lemma \ref{lem:plane2-connected}, 
$H^\times$ is $2$-connected
and 
$\delta(H^\times)\ge 2$.
By Claim~\ref{claim:wv}
and Corollary~\ref{cor:lowdegree-special} (ii),
$n_2(H^\times) +n_3(H^\times)
\ge 7$.
Since each fake vertex in 
$H^\times$ is of degree $4$
and $\partial(F_1^*)\cup 
\partial(F_2^*)$ contains 
six true vertices only, 
$H^\times$ has a vertex $p$
which is 
not on $\partial(F^*_1)\cup 
\partial(F^*_2)$ with the property that $d_{H^\times}(p)=2$ or $3$.
Clearly, $p$ is a true vertex in 
$H^\times$.

Since 
$p$ is a true vertex in $H^\times$ which 
is not on $\partial(F^*_1)\cup 
\partial(F^*_2)$, 
$p$ is incident 
with alternating $4$-faces of $H^\times$ only.
Since $G\in \Maxe{n}{3}$, 
we have $H\in \Maxe{n}{3}$.
By Lemma \ref{lem:degree3po},
$d_{H^\times}(p)=d_H(p)\ge 4$,
contradicting the fact that 
$d_{H^\times}(p)\le 3$.

Hence Proposition~\ref{lem:3n-7} holds.
\end{proof}

\noindent\textbf{Proof of Theorem~\ref{thm:K3}.} Combining 
Corollary~\ref{cor-3-3}, 
Lemma~\ref{lem:3n-6}
and Proposition \ref{lem:3n-7}, the upper bound $3n-8$ of $\maxe{n}{3}$  follows.  
Furthermore, Karpov  \cite{zbMATH06347737} showed that  for every even $n\ge 8$ there exist $n$-vertex bipartite 1-planar graphs with $3n-8$ edges, and these graphs are $K_3$-free. Therefore, Theorem \ref{thm:K3} holds.

\section{Proof of Theorem \ref{thm:K4}} 

It is well known that any drawing of a (simple) graph $G$ with $n\ge 3$ vertices and $m$ edges in the plane has at least $ m-(3n-6)$ crossings
(see~ page 10 in \cite{MR3751397}). 
We first obtain a sharper 
bound on the number of crossings in any 1-planar drawing of a graph of $\Maxe{n}{4}$.

\begin{lem}
\label{lem:k4free}
For any connected 1-plane graph 
$G$ in $\Maxe{n}{4}$,
where $n\ge 4$, 
$G$ has at least $2(m-3n+6)$ crossings,
where $m=e(G)$.
\end{lem}

\begin{proof}
Let $c$ be the number of crossings in  $G$.  By Lemma \ref{obs:A-nonneg_K4}, $A\ge -c$. Moreover, $B\ge 0$ and $C\ge 0$  by Lemma \ref{obs:pB-nonneg}. Now Lemma~\ref{lem:maintool} gives
\[
m=3n-6-\frac{A+B}{2}-\frac34\,C
\le 3n-6-\frac{A}{2}
\le 3n-6+\frac{c}{2}.
\]
which yields $c\ge 2(m-3n+6)$, as required.
\end{proof}

\begin{prop}
	\label{k4free-2}
	$\maxe{n}{4}\le  \frac 72 n-7$
	for $n\ge 4$.
\end{prop}

\begin{proof}
Let $G\in \Maxe{n}{4}$
with $e(G)=\maxe{n}{4}$,
where $n\ge 4$, 
and let 
$c$ denote the number of crossings in a $1$-planar drawing of $G$.
By the assumption of $G$,
$G$ is connected.

By Czap and Hud\'ak~\cite{MR3084596}, every $1$-planar drawing of a $1$-planar graph on $n$ vertices has at most $n-2$ crossings, i.e., 
$c\le n-2$.
On the other hand, Lemma~\ref{lem:k4free} yields
$c\ge 2(m-3n+6)$,
where $m=e(G)$.
Combining these two inequalities implies that 
$
2(m-3n+6)\le n-2,
$
and therefore
$m\le \frac{7}{2}n-7$.
\end{proof}

We are now going to 
complete the proof of 
Theorem~\ref{thm:K4}.

\vspace{2 mm}

\noindent {\it Proof} of Theorem~\ref{thm:K4}:
For $1\le n\le 7$, 
Tur\'an's theorem gives
$e\left(T_3(n)\right)=\Bigl\lfloor \frac{n^2}{3}\Bigr\rfloor,
$
where $T_3(n)$ is planar
for all $1\le n\le 6$
and  $T_3(7)
\cong K_{3,2,2}$ is $1$-planar
(see \cite{MR2876333}).
Hence, for  $1\le n\le 7$, 
$\maxe{n}{4}=
\bigl\lfloor n^2/3\bigr\rfloor$.

For $n=8$, Tur\'an's theorem
yields that 
$T_3(8)\cong K_{3,3,2}$
is the only $K_4$-free graph
of order $8$ and size $21$.
However, $K_{3,3,2}$ is not $1$-planar (see \cite{MR2876333}),
implying that 
$\maxe{n}{4}\le 20$.
Furthermore,  
the graph shown in 
Figure \ref{fig:two-ex} (a) belongs to $\Maxe{8}{4}$
with size $20$.
Thus, $\maxe{8}{4}=20=\bigl\lfloor n^2/3\bigr\rfloor
-1$.
Hence, 
Theorem~\ref{thm:K4}
holds for $1\le n\le 8$.
\begin{figure}[H]
	\centering
	\begin{tikzpicture}[scale=0.65]
		\begin{pgfonlayer}{nodelayer}
			\node [style=whitenode] (0) at (-6, -1.97213) {};
			\node [style=whitenode] (2) at (-6.01557, 2.01885) {};
			\node [style=whitenode] (4) at (-1.99713, 2.02541) {};
			\node [style=whitenode] (9) at (-2.99713, 1.04508) {};
			\node [style=whitenode] (14) at (-5.06434, 1.02664) {};
			\node [style=whitenode] (19) at (-5.01967, -0.778689) {};
			\node [style=whitenode] (21) at (-2.9418, -0.760246) {};
			\node [style=whitenode] (23) at (-2.01147, -1.9668) {};
			
			\node [style=whitenode] (24) at (8.03349, -2.001) {};
			\node [style=whitenode] (25) at (5.0014, 0.96057) {};
			\node [style=whitenode] (26) at (7.01348, -0.24013) {};
			\node [style=whitenode] (27) at (5.01348, -0.27395) {};
			\node [style=whitenode] (28) at (8.05241, 1.95529) {};
			\node [style=whitenode] (29) at (4.00287, 1.95605) {};
			\node [style=whitenode] (30) at (6.01114, -0.5) {};
			\node [style=whitenode] (31) at (7.01329, 0.9705) {};
			\node [style=whitenode] (32) at (4.01605, -1.99419) {};
		\end{pgfonlayer}
		\begin{pgfonlayer}{edgelayer}
			\draw (4) to (9);
			\draw (4) to (21);
			\draw (0) to (2);
			\draw (14) to (19);
			\draw (19) to (21);
			\draw (2) to (14);
			\draw (2) to (4);
			\draw (4) to (23);
			\draw (0) to (19);
			\draw (2) to (19);
			\draw (21) to (23);
			\draw (9) to (21);
			\draw [style=blueedge, bend left=75, looseness=2.00] (2) to (23);
			\draw [style=blueedge, bend right=75, looseness=2.00] (4) to (0);
			\draw [style=blueedge] (19) to (23);
			\draw [style=blueedge] (21) to (0);
			\draw [style=blueedge] (14) to (21);
			\draw [style=blueedge] (9) to (19);
			\draw [style=blueedge] (2) to (9);
			\draw [style=blueedge] (4) to (14);
			\draw (24) to (26);
			\draw (24) to (28);
			\draw (24) to (32);
			\draw (25) to (27);
			\draw (25) to (29);
			\draw (25) to (31);
			\draw (26) to (30);
			\draw (26) to (31);
			\draw (27) to (30);
			\draw (27) to (32);
			\draw (28) to (29);
			\draw (28) to (31);
			\draw (29) to (32);
			\draw (30) to (31);
			\draw (30) to (32);
			\draw [style=blueedge] (29) to (27);
			\draw [style=blueedge] (25) to (32);
			\draw [style=blueedge] (27) to (31);
			\draw [style=blueedge] (25) to (30);
			\draw [style=blueedge] (32) to (26);
			\draw [style=blueedge] (30) to (24);
			\draw [style=blueedge] (29) to (31);
			\draw [style=blueedge] (25) to (28);
			\draw [style=blueedge] (28) to (26);
			\draw [style=blueedge] (31) to (24);
			\draw [style=blueedge, bend left=75, looseness=2.0] (32) to (28);
			\draw [style=blueedge, bend left=75, looseness=2.0] (29) to (24);
		\end{pgfonlayer}
	\end{tikzpicture}
	
	(a) A graph in $\Maxe{8}{4}$ \hspace{1.5 cm} (b) A graph in $\Maxe{9}{5}$
	
	\caption{ Graphs in $\Maxe{8}{4}$ 
		and $\Maxe{9}{5}$,
	respectively}
	\label{fig:two-ex}
\end{figure}

Now assume that $n\ge 9$.
By Proposition~\ref{k4free-2},
it suffices to show that for any $n\ge 9$, there exists 
$G_n \in \Maxe{n}{4}$
with $e(G_n)= \left\lfloor \frac{7}{2}n \right\rfloor -7$.
Let $G_9$ and $G_{10}$ be the 
graphs shown in Figure~\ref{fig:HkG9G10}
(ii) and (iii), respectively.
It can be verified that 
$G_n\in \Maxe{n}{4}$
for $n=9,10$.
Note that 
$e(G_9)=24$ and 
$e(G_{10})=28$,
implying that $e(G_n)=\left\lfloor \frac{7}{2}n \right\rfloor -7$
for $n=9,10$.
Thus, the following claim holds.

\begin{claim}
	\label{cl:k4free-0}
	For $n=9,10$, 
	$G_n\in \Maxe{n}{4}$ 
	with $e(G_n)=\left\lfloor \frac{7}{2}n \right\rfloor -7$.
\end{claim}

\begin{figure}[h!]
	\centering
	
	\makebox[\textwidth][l]{%
		\begin{minipage}[t]{0.3\textwidth}
			\centering
			\begin{tikzpicture}[scale=0.46]
				\begin{pgfonlayer}{nodelayer}
					\node [style=whitenode] (26) at (-4, 2.25) {$a_1$};
					\node [style=whitenode] (27) at (-4, 0.25) {$b_1$};
					\node [style=whitenode] (28) at (-2, 2.25) {};
					\node [style=whitenode] (29) at (-2, 0.25) {};
					\node [style=whitenode] (30) at (0.05, 0.25) {};
					\node [style=whitenode] (31) at (0, 2.25) {};
					\node [style=whitenode] (32) at (2, 2.25) {};
					\node [style=whitenode] (33) at (2, 0.25) {};
					\node [style=whitenode] (34) at (4, 2.25) {$a_{k}$};
					\node [style=whitenode] (35) at (4, 0.25) {$b_{k}$};
					\node [style=blacknode] (36) at (0, 5.25) {};
					\node [style=blacknode] (37) at (0, -2.75) {};
					\node [style=none] (39) at (0, -3.75) {$v$};
					\node [style=none] (40) at (1, 2.25) {$\ldots$};
					\node [style=none] (41) at (1, 0.25) {$\ldots$};
					\node [style=none] (u) at (0, 5.75) {$u$};
				\end{pgfonlayer}
				\begin{pgfonlayer}{edgelayer}
					\draw (26) to (28);
					\draw (28) to (31);
					\draw (32) to (34);
					\draw (34) to (35);
					\draw (35) to (33);
					\draw (30) to (29);
					\draw (29) to (27);
					\draw (27) to (26);
					\draw [bend right=15, looseness=0.75] (36) to (26);
					\draw (36) to (28);
					\draw (36) to (31);
					\draw (36) to (32);
					\draw [bend left=15] (36) to (34);
					\draw [bend right=15] (27) to (37);
					\draw (37) to (29);
					\draw (37) to (30);
					\draw (37) to (33);
					\draw [bend right=15] (37) to (35);
					\draw [style=blueedge, bend left=15] (37) to (26);
					\draw [style=blueedge, bend left=15] (28) to (37);
					\draw [style=blueedge, in=285, out=45, looseness=1.50] (37) to (32);
					\draw [style=blueedge, bend right=45] (36) to (29);
					\draw [style=blueedge, bend left] (30) to (36);
					\draw [style=blackedge] (28) to (29);
					\draw [style=blackedge] (31) to (30);
					\draw [style=blackedge] (32) to (33);
					\draw [style=blueedge, bend left=15] (36) to (35);
				\end{pgfonlayer}
			\end{tikzpicture}

		\end{minipage}
		\hspace{0.02\textwidth}%
		
		\begin{minipage}[t]{0.64\textwidth}
			\centering
\begin{tikzpicture}[scale=0.46, bezier bounding box]
	\begin{pgfonlayer}{nodelayer}
		\node [style=whitenode] (24) at (2.1988, 2) {$x$};
		\node [style=whitenode] (25) at (2.1988, 0) {$y$};
		\node [style=whitenode] (26) at (4.1988, 2) {$a_1$};
		\node [style=whitenode] (27) at (4.1988, 0) {$b_1$};
		\node [style=whitenode] (28) at (6.2488, 0) {$b_2$};
		\node [style=whitenode] (29) at (6.1988, 2) {$a_2$};
		\node [style=whitenode] (30) at (8.1988, 2) {$x'$};
		\node [style=whitenode] (31) at (8.1988, 0) {$y'$};
		\node [style=blacknode] (32) at (5.1988, 5) {};
		\node [style=blacknode] (33) at (5.1988, -3) {};
		\node [style=none] (34) at (5.2, 5.5) {$u$};
		\node [style=none] (35) at (5.2, -3.6) {$v$};
		\node [style=whitenode] (39) at (-7, 0) {$y$};
		\node [style=whitenode] (40) at (-5, 2) {$a_1$};
		\node [style=whitenode] (41) at (-3, 0) {$b_2$};
		\node [style=whitenode] (42) at (-9, 0) {$z$};
		\node [style=whitenode] (43) at (-5, 0) {$b_1$};
		\node [style=whitenode] (44) at (-7, 2) {$x$};
		\node [style=whitenode] (45) at (-3, 2) {$a_2$};
		\node [style=blacknode] (46) at (-4, 5) {};
		\node [style=blacknode] (47) at (-4, -3) {};
		\node [style=none] (48) at (-4, 5.5) {$u$};
		\node [style=none] (49) at (-4, -3.6) {$v$};
	\end{pgfonlayer}
	\begin{pgfonlayer}{edgelayer}
		\draw (24) to (26);
		\draw (26) to (29);
		\draw (30) to (31);
		\draw (28) to (27);
		\draw (27) to (25);
		\draw (25) to (24);
		\draw [bend right, looseness=0.75] (32) to (24);
		\draw [style={shadow_silver}] (27.center)
			 to (26.center)
			 to (32.center)
			 to (29.center)
			 to (28.center)
			 to (33.center)
			 to cycle;
		\draw [bend left] (32) to (30);
		\draw [bend right] (25) to (33);
		\draw [bend right=45] (33) to (31);
		\draw [bend left=15] (33) to (24);
		\draw [style=blueedge, bend left=15] (26) to (33);
		\draw [in=15, out=-30, looseness=0.75] (29) to (33);
		\draw [bend right=45] (32) to (27);
		\draw [style=blueedge, in=-90, out=135] (28) to (32);
		\draw [bend left=15, looseness=0.25] (32) to (31);
		\draw (29) to (30);
		\draw (28) to (31);
		\draw [bend right=60, looseness=1.25] (32) to (25);
		\draw [in=-60, out=-15, looseness=1.50] (33) to (30);
		\draw [bend left=90, looseness=2.25] (24) to (30);
		\draw [bend right=90, looseness=2.25] (25) to (31);
		\draw (26) to (29);
		\draw (27) to (28);
		\draw (39) to (43);
		\draw [style={shadow_silver}] (47.center)
			 to (41.center)
			 to (45.center)
			 to (46.center)
			 to (40.center)
			 to (43.center)
			 to cycle;
		\draw (39) to (44);
		\draw (40) to (44);
		\draw (42) to (44);
		\draw [bend right=105, looseness=3.50] (39) to (45);
		\draw [in=-45, out=-75, looseness=3.25] (42) to (45);
		\draw [in=180, out=150, looseness=1.50] (39) to (46);
		\draw [style=blueedge, in=-90, out=135, looseness=0.75] (41) to (46);
		\draw [bend left=60, looseness=1.25] (42) to (46);
		\draw [in=240, out=140, looseness=1.20] (43) to (46);
		\draw (44) to (46);
		\draw (39) to (47);
		\draw [style=blueedge, bend left=15, looseness=0.75] (40) to (47);
		\draw (42) to (47);
		\draw [in=120, out=-75] (44) to (47);
		\draw [bend left=45, looseness=1.25] (45) to (47);
		\draw [style=blackedge] (40) to (45);
		\draw [style=blackedge] (43) to (41);
	\end{pgfonlayer}
\end{tikzpicture}

		\end{minipage}%
	}

(i) $H_k$ \hspace{4 cm} 
(ii) $G_9$ \hspace{4 cm} 
(iii) $G_{10}$

	\caption{Graphs $H_k$, $G_9$, and $G_{10}$.}
	\label{fig:HkG9G10}
\end{figure}
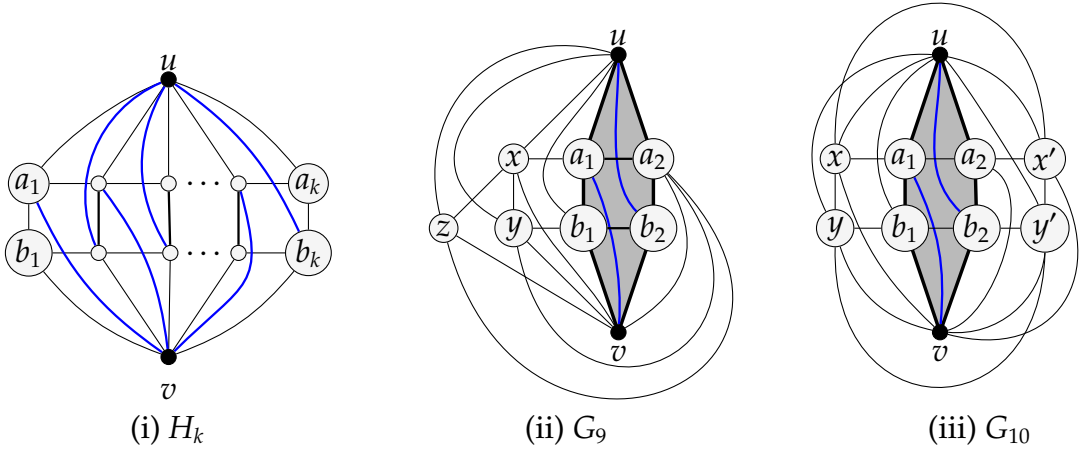

For $i=1,2$, let 
$Q_i$ denote the $1$-plane graph obtained from $G_{8+i}$ by removing all edges in 
the set $\{a_1a_2,b_1b_2, ub_2,va_1\}$.
Then, $Q_i$ has 
a true $6$-face $F_i$ whose boundary 
is the cycle $ua_1b_1vb_2a_2u$,
i.e., the shaded region 
in 
Figure~\ref{fig:HkG9G10} (ii)
and (iii), respectively.


The \emph{ladder graph} $L_k$ is the Cartesian product
$L_k := P_k \square K_2$, 
where $k\ge 1$.
Equivalently, 
$L_k$ is the graph with 
vertex set $\{a_i, b_i: 1\le i\le k\}$
and edge set 
$\{a_ia_{i+1}, b_ib_{i+1}:1\le i\le k-1 \}\cup \{a_ib_i: 1\le i\le k\}$.
The joint of two vertex-disjoint
graphs $G$ and $H$, denoted by  $G \vee H$, 
is the graph obtained from their union by adding all edges $xy$
for $x\in V(G)$ and $y\in V(H)$.
Let $H_k$ denote the graph
$L_k \vee N_2-\{ub_1, va_k\}$,
where $V(N_2)=\{u,v\}$, 
as shown in  Figure~\ref{fig:HkG9G10}(i).
Clearly, 
$$
|V(L_k)|=2k, 
\quad  |E(L_k)|=3k-2, 
\quad |V(H_k)|=2k+2,
\quad
|E(H_k)|=7k-4.
$$

By the definition of $Q_1$, $Q_2$ and $H_k$, we have the following conclusion.

\begin{claim}
	\label{cl:k4free-1}
Both	$Q_1$ and $Q_2$ are $K_4$-free, and 
	$H_k+\{ub_1,va_k\}$ is $K_4$-free
	for all $k\ge 1$.
\end{claim}

For any $n=2\ell+1$,
where $\ell\ge 5$, 
let $G_n$ be the $1$-plane graph obtained from $Q_1$
and $H'_{\ell-2}:=H_{\ell-2}
-\{ua_1, a_1b_1,b_1v, vb_{\ell-2},
b_{\ell-2}a_{\ell-2}, a_{\ell-2}u\}
$
by inserting $H'_{\ell-2}$ into 
the face $F_1$ of $Q_1$ 
such that 
the six vertices $u, a_1,b_1,v, b_{\ell-2}$ and $a_{\ell-2}$  in $H'_{\ell-2}$ 
are identified with 
the six vertices $u, a_1,b_1,v, b_{2}$ and $a_{2}$ in $Q_1$,
respectively, 
and 
as shown in 
Figure \ref{fig:G2L2L+1} (i).

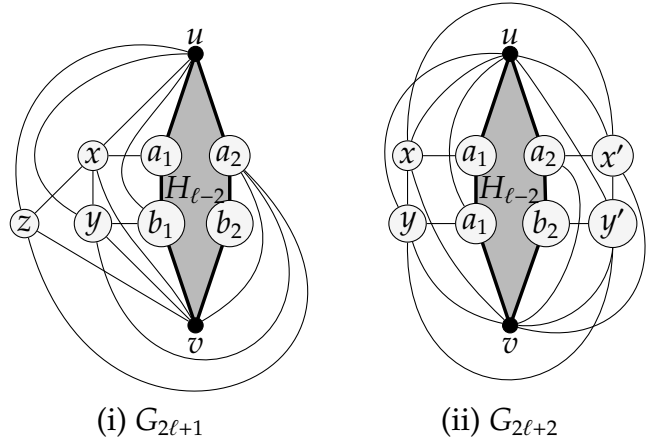
\begin{figure}[H]
	\centering
	\begin{tikzpicture}[scale=0.45, bezier bounding box]
		\begin{pgfonlayer}{nodelayer}
			\node [style=whitenode] (24) at (2.1988, 2) {$x$};
			\node [style=whitenode] (25) at (2.1988, 0) {$y$};
			\node [style=whitenode] (26) at (4.1988, 2) {$a_1$};
			\node [style=whitenode] (27) at (4.1988, 0) {$a_1$};
			\node [style=whitenode] (28) at (6.2488, 0) {$b_2$};
			\node [style=whitenode] (29) at (6.1988, 2) {$a_2$};
			\node [style=whitenode] (30) at (8.1988, 2) {$x'$};
			\node [style=whitenode] (31) at (8.1988, 0) {$y'$};
			\node [style=blacknode] (32) at (5.1988, 5) {};
			\node [style=blacknode] (33) at (5.1988, -3) {};
			\node [style=none] (34) at (5.2, 5.5) {$u$};
			\node [style=none] (35) at (5.2, -3.6) {$v$};
			\node [style=whitenode] (39) at (-7, 0) {$y$};
			\node [style=whitenode] (40) at (-5, 2) {$a_1$};
			\node [style=whitenode] (41) at (-3, 0) {$b_2$};
			\node [style=whitenode] (42) at (-9, 0) {$z$};
			\node [style=whitenode] (43) at (-5, 0) {$b_1$};
			\node [style=whitenode] (44) at (-7, 2) {$x$};
			\node [style=whitenode] (45) at (-3, 2) {$a_2$};
			\node [style=blacknode] (46) at (-4, 5) {};
			\node [style=blacknode] (47) at (-4, -3) {};
			\node [style=none] (48) at (-4, 5.5) {$u$};
			\node [style=none] (49) at (-4, -3.6) {$v$};
			\node [style=none] (51) at (-4, 1) {$H_{\ell-2}$};
			\node [style=none] (52) at (5.1988, 1) {$H_{\ell-2}$};
		\end{pgfonlayer}
		\begin{pgfonlayer}{edgelayer}
			\draw (24) to (26);
			\draw (26) to (29);
			\draw (30) to (31);
			\draw (28) to (27);
			\draw (27) to (25);
			\draw (25) to (24);
			\draw [bend right, looseness=0.75] (32) to (24);
			\draw [style={shadow_silver}] (27.center)
			to (26.center)
			to (32.center)
			to (29.center)
			to (28.center)
			to (33.center)
			to cycle;
			\draw [bend left] (32) to (30);
			\draw [bend right] (25) to (33);
			\draw [bend right=45] (33) to (31);
			\draw [bend left=15] (33) to (24);
			\draw [in=15, out=-30, looseness=0.75] (29) to (33);
			\draw [bend right=45] (32) to (27);
			\draw [bend left=15, looseness=0.25] (32) to (31);
			\draw (29) to (30);
			\draw (28) to (31);
			\draw [bend right=60, looseness=1.25] (32) to (25);
			\draw [in=-60, out=-15, looseness=1.50] (33) to (30);
			\draw [bend left=90, looseness=2.25] (24) to (30);
			\draw [bend right=90, looseness=2.25] (25) to (31);
			\draw (39) to (43);
			\draw [style={shadow_silver}] (47.center)
			to (41.center)
			to (45.center)
			to (46.center)
			to (40.center)
			to (43.center)
			to cycle;
			\draw (39) to (44);
			\draw (40) to (44);
			\draw (42) to (44);
			\draw [bend right=105, looseness=3.50] (39) to (45);
			\draw [in=-45, out=-75, looseness=3.25] (42) to (45);
			\draw [in=180, out=150, looseness=1.50] (39) to (46);
			\draw [bend left=60, looseness=1.25] (42) to (46);
			\draw [in=240, out=140, looseness=1.20] (43) to (46);
			\draw (44) to (46);
			\draw (39) to (47);
			\draw (42) to (47);
			\draw [in=120, out=-75] (44) to (47);
			\draw [bend left=45, looseness=1.25] (45) to (47);
		\end{pgfonlayer}
	\end{tikzpicture}

(i) $G_{2\ell+1}$\hspace{3 cm}
(ii) $G_{2\ell+2}$
	
	\caption{Graphs $G_{2\ell+1}$ and $G_{2\ell+2}$ for $\ell\ge 5$}
	\label{fig:G2L2L+1}
\end{figure}

Similarly, 
for any $n=2\ell+2$,
where $\ell\ge 5$, 
let $G_n$ be the $1$-plane graph obtained from $Q_2$
and 
$H'_{\ell-2}:=H_{\ell-2}
-\{ua_1, a_1b_1,b_1v, vb_{\ell-2},
b_{\ell-2}a_{\ell-2}, a_{\ell-2}u\}
$
by inserting $H'_{\ell-2}$ into 
the face $F_2$ of $Q_2$ 
such that 
the six vertices $u, a_1,b_1,v, b_{\ell-2}$ and $a_{\ell-2}$  in $H'_{\ell-2}$ 
are identified with 
the six vertices $u, a_1,b_1,v, b_{2}$ and $a_{2}$ in $Q_2$,
respectively, as shown in 
Figure \ref{fig:G2L2L+1} (ii).

\begin{claim}
	\label{cl:k4free}
For every $n\ge 11$, 
$G_n\in \Maxe{n}{4}$ 
with $e(G_n)=\left\lfloor \frac{7}{2}n \right\rfloor -7$.
\end{claim}

\begin{proof} 
	Let $\ell\ge 5$ and 
	$n=2\ell+1$.
	Clearly, $G_{n}$  is 
	$1$-plane.

Suppose that $G_{n}$ 
contains a subgraph 
$W$ which is isomorphic to $K_4$.
By the definition of $G_n$, 
it is clear that either 
$V(W)\subseteq V(H_{\ell-2})$
or $V(W)\subseteq V(Q_1)$,
implying that either 
$H_{\ell-2}+\{ub_1,va_{\ell-2}\}$ or $Q_1$ is 
not $K_4$-free,
a contradiction to Claim~\ref{cl:k4free-1}.

Hence $G_n\in \Maxe{n}{4}$.
Similarly, 
$G_n\in \Maxe{n}{4}$ for 
$n=2\ell+2$, where $\ell\ge 5$.

Note that, for $\ell\ge 5$, we have 
\Eqnn{
e(G_{2\ell+1})
&=& e(H_{\ell-2})+e(G_9)-10
= \left(7 (\ell-2)-4\right)+24-10
= 7\ell-4,\\
e(G_{2\ell+2})
&=&e(H_{\ell-2})+e(G_{10})-10
= \left(7(\ell-2)-4\right)+28-10
= 7\ell,
}
implying that 
$e(G_n)=\left\lfloor \frac{7}{2}n \right\rfloor-7$
for all $n\ge 11$.
Thus Claim~\ref{cl:k4free} holds.
\end{proof}

By Proposition~\ref{k4free-2}
and Claims~\ref{cl:k4free-0}
and~\ref{cl:k4free}, 
Theorem~\ref{thm:K4}
holds for $n\ge 9$.
\hfill $\Box$

\begin{remark}
When $2\ell-2 $ is a multiple of $4$, the graph $G_{2\ell}$ coincides with the $3$-partite $1$-planar graph  with $\frac{7}{2}(2\ell) -7=7\ell-7$ edges,  due to Suzuki~\cite{MR4305146}.  Our construction in the proof
of Theorem~\ref{thm:K4} above
is inspired by Suzuki’s example.
\end{remark}

\section{Proof of Theorem \ref{thm:K5}} 

We first prove Theorem \ref{thm:K5}
for $1\le n\le 7$ or $n=9$.

\begin{prop}\label{prop7-1}
For $1\le n\le 7$ or $n=9$,
$\maxe{n}{5}=	
\Floor{\frac {3n^2}{8}}
-3\Floor{\frac n9}$.
\end{prop}

\begin{proof}
For $1\le n\le 7$,  $T_4(n)$ is a 
subgraph of the complete $4$-partite graph $K_{2,2,2,1}$
which is known to be 1-planar \cite{MR2876333}. Consequently, $T_4(n)$ is $1$-planar, implying that 
$\maxe{n}{5}
=e(T_4(n))=
\left \lfloor \frac {3n^2}{8}\right 
\rfloor$.
 
For $n=9$ and any 
$G\in \Maxe{n}{5}$, 
by Lemma \ref{lem:edges_1planar}, we have $e(G)\le 4n-9=27=
\left \lfloor \frac {3n^2}{8}\right 
\rfloor
-3\Floor{\frac n9}$. 
Moreover, the graph in 
Figure~\ref{fig:two-ex} (b)
belongs to $\Maxe{9}{5}$
with size $27$.
Hence $\maxe{9}{5}=27
=\left \lfloor \frac {3n^2}{8}\right 
\rfloor
-3\Floor{\frac n9}.
$
\end{proof} 

For $n=8$ or $n\ge 10$,
by Lemma \ref{lem:edges_1planar},
any graph $G\in \Maxe{n}{5}$ 
contains at most $4n-8$ edges.
Thus, in order to show that 
$\maxe{n}{5}=4n-8$,  
it suffices to show 
the existence of a graph 
$G_n\in \Maxe{n}{5}$ 
with $e(G_n)=4n-8$.

For $n\in \{8,10,11,13\}$,
let $G_n$ be the graph 
shown in Figure~\ref{fig:K5free-ex}.

\begin{prop}\label{prop7-2}
	For $n\in \{8,10,11,13\}$,
	$G_n\in \Maxe{n}{5}$
	with $e(G_n)=4n-8$.
\end{prop}

\begin{proof}
It can be verified easily 
from Figure~\ref{fig:K5free-ex}
that for any  $n\in \{8,10,11,13\}$,
$G_n\in \Maxe{n}{5}$
with $e(G_n)=4n-8$.
Thus, Proposition~\ref{prop7-2}
holds.
\end{proof}

\begin{figure}[H]
	\centering
	
	\begin{subfigure}[b]{0.24\textwidth}
		\centering
		\scalebox{0.52}{
			\begin{tikzpicture}
				\begin{pgfonlayer}{nodelayer}
					\node [style=whitenode] (0) at (-2, 2) {};
					\node [style=whitenode] (1) at (2, 2) {};
					\node [style=whitenode] (2) at (-2, -2) {};
					\node [style=whitenode] (3) at (2, -2) {};
					\node [style=whitenode] (4) at (2, 2) {};
					\node [style=whitenode] (5) at (-1, 1) {};
					\node [style=whitenode] (6) at (1, 1) {};
					\node [style=whitenode] (7) at (-1, -1) {};
					\node [style=whitenode] (8) at (1, -1) {};
				\end{pgfonlayer}
				\begin{pgfonlayer}{edgelayer}
					\draw [style=blackedge] (0) to (4);
					\draw [style=blackedge] (4) to (3);
					\draw [style=blackedge] (3) to (2);
					\draw [style=blackedge] (2) to (0);
					\draw [style=blackedge] (0) to (5);
					\draw [style=blackedge] (5) to (6);
					\draw [style=blackedge] (6) to (4);
					\draw [style=blackedge] (5) to (7);
					\draw [style=blackedge] (7) to (2);
					\draw [style=blackedge] (7) to (8);
					\draw [style=blackedge] (8) to (3);
					\draw [style=blackedge] (8) to (6);
					\draw [style=blueedge] (0) to (6);
					\draw [style=blueedge] (4) to (5);
					\draw [style=blueedge] (5) to (8);
					\draw [style=blueedge] (6) to (7);
					\draw [style=blueedge] (7) to (3);
					\draw [style=blueedge] (8) to (2);
					\draw [style=blueedge] (2) to (5);
					\draw [style=blueedge] (7) to (0);
					\draw [style=blueedge] (6) to (3);
					\draw [style=blueedge] (8) to (4);
					\draw [style=blueedge, bend left=75, looseness=2.00] (2) to (4);
					\draw [style=blueedge, bend left=75, looseness=2.00] (0) to (3);
				\end{pgfonlayer}
			\end{tikzpicture}
		}
		\caption{$G_8$}
	\end{subfigure}\hfill
	\begin{subfigure}[b]{0.24\textwidth}
		\centering
		\scalebox{0.52}{
			\begin{tikzpicture}
				\begin{pgfonlayer}{nodelayer}
					\node [style=whitenode] (1) at (-2, 3) {};
					\node [style=whitenode] (2) at (-2, -1) {};
					\node [style=whitenode] (3) at (2, 3) {};
					\node [style=whitenode] (4) at (-1.06194, 2.00001) {};
					\node [style=whitenode] (5) at (-1.0902, 0.50001) {};
					\node [style=whitenode] (6) at (0.0098, -0.02499) {};
					\node [style=whitenode] (7) at (2.025, -1.025) {};
					\node [style=whitenode] (8) at (1.05436, 2.00001) {};
					\node [style=whitenode] (9) at (0.04132, 0.825008) {};
					\node [style=whitenode] (10) at (1.01305, 1e-05) {};
				\end{pgfonlayer}
				\begin{pgfonlayer}{edgelayer}
					\draw (1) to (2);
					\draw (1) to (3);
					\draw (1) to (4);
					\draw (2) to (5);
					\draw (4) to (5);
					\draw (2) to (6);
					\draw (2) to (7);
					\draw (3) to (7);
					\draw (3) to (8);
					\draw (4) to (8);
					\draw (5) to (9);
					\draw (6) to (9);
					\draw (8) to (9);
					\draw (6) to (10);
					\draw (7) to (10);
					\draw (8) to (10);
					\draw [style=blueedge] (1) to (5);
					\draw [style=blueedge] (4) to (2);
					\draw [style=blueedge] (2) to (9);
					\draw [style=blueedge] (5) to (6);
					\draw [style=blueedge] (5) to (8);
					\draw [style=blueedge] (4) to (9);
					\draw [style=blueedge] (9) to (10);
					\draw [style=blueedge] (8) to (6);
					\draw [style=blueedge] (1) to (8);
					\draw [style=blueedge] (3) to (4);
					\draw [style=blueedge] (8) to (7);
					\draw [style=blueedge] (3) to (10);
					\draw [style=blueedge] (2) to (10);
					\draw [style=blueedge] (6) to (7);
					\draw [style=blueedge, bend left=75,looseness=2.00] (2) to (3);
					\draw [style=blueedge, bend left=75,looseness=2.00] (1) to (7);
				\end{pgfonlayer}
			\end{tikzpicture}
		}
		\caption{$G_{10}$}
	\end{subfigure}\hfill
	\begin{subfigure}[b]{0.24\textwidth}
		\centering
		\scalebox{0.52}{
			\begin{tikzpicture}[scale=1.0]
				\begin{pgfonlayer}{nodelayer}
					\node [style=whitenode] (1) at (2, 2) {};
					\node [style=whitenode] (2) at (2, -2) {};
					\node [style=whitenode] (3) at (-2, 2) {};
					\node [style=whitenode] (4) at (1.19967, 1.23202) {};
					\node [style=whitenode] (5) at (1.05372, 0.31785) {};
					\node [style=whitenode] (6) at (0.28551, -0.281418) {};
					\node [style=whitenode] (7) at (0.02048, -1.21748) {};
					\node [style=whitenode] (8) at (-2, -2) {};
					\node [style=whitenode] (9) at (-1.27118, -0.24248) {};
					\node [style=whitenode] (10) at (-1.0797, 1.25319) {};
					\node [style=whitenode] (11) at (0.21151, 0.24654) {};
					\node [style=whitenode] (12) at (-0.47997, -0.290799) {};
					\node [style=whitenode] (13) at (-1.30857, -1.16166) {};
				\end{pgfonlayer}
				\begin{pgfonlayer}{edgelayer}
					\draw (1) to (2);
					\draw (1) to (3);
					\draw (1) to (4);
					\draw (2) to (5);
					\draw (4) to (5);
					\draw (2) to (6);
					\draw (2) to (7);
					\draw (2) to (8);
					\draw (3) to (8);
					\draw (3) to (9);
					\draw (3) to (10);
					\draw (4) to (10);
					\draw (5) to (11);
					\draw (6) to (11);
					\draw (10) to (11);
					\draw (6) to (12);
					\draw (7) to (12);
					\draw (9) to (12);
					\draw (10) to (12);
					\draw (7) to (13);
					\draw (8) to (13);
					\draw (9) to (13);
					\draw [style=blueedge] (3) to (4);
					\draw [style=blueedge] (1) to (10);
					\draw [style=blueedge] (3) to (12);
					\draw [style=blueedge] (10) to (9);
					\draw [style=blueedge] (3) to (13);
					\draw [style=blueedge] (9) to (8);
					\draw [style=blueedge] (13) to (12);
					\draw [style=blueedge] (9) to (7);
					\draw [style=blueedge] (7) to (8);
					\draw [style=blueedge] (13) to (2);
					\draw [style=blueedge] (2) to (12);
					\draw [style=blueedge] (7) to (6);
					\draw [style=blueedge] (12) to (11);
					\draw [style=blueedge] (10) to (6);
					\draw [style=blueedge] (4) to (11);
					\draw [style=blueedge] (10) to (5);
					\draw [style=blueedge] (11) to (2);
					\draw [style=blueedge] (5) to (6);
					\draw [style=blueedge] (5) to (1);
					\draw [style=blueedge] (4) to (2);
					\draw [style=blueedge, bend left=75,looseness=2.00] (8) to (1);
					\draw [style=blueedge, bend left=75,looseness=2.00] (3) to (2);
				\end{pgfonlayer}
			\end{tikzpicture}
		}
		\caption{$G_{11}$}
	\end{subfigure}\hfill
	\begin{subfigure}[b]{0.24\textwidth}
		\centering
		\scalebox{0.52}{
			\begin{tikzpicture}[scale=1.0]
				\begin{pgfonlayer}{nodelayer}
					\node [style=whitenode] (1) at (2, 2) {};
					\node [style=whitenode] (2) at (2, -2) {};
					\node [style=whitenode] (3) at (-2, 2) {};
					\node [style=whitenode] (4) at (1.25, 1.00002) {};
					\node [style=whitenode] (5) at (1.28571, -0.49998) {};
					\node [style=whitenode] (6) at (0.23571, -1.32498) {};
					\node [style=whitenode] (7) at (-2, -2) {};
					\node [style=whitenode] (8) at (-1.28571, 0.500019) {};
					\node [style=whitenode] (9) at (-0.53571, 1.25002) {};
					\node [style=whitenode] (10) at (0.025, 2.82055e-05) {};
					\node [style=whitenode] (11) at (0.53214, -0.574978) {};
					\node [style=whitenode] (12) at (-1.25, -1.24998) {};
					\node [style=whitenode] (13) at (-0.60714, 0.500022) {};
				\end{pgfonlayer}
				\begin{pgfonlayer}{edgelayer}
					\draw (1) to (2);
					\draw (1) to (3);
					\draw (1) to (4);
					\draw (2) to (5);
					\draw (4) to (5);
					\draw (2) to (6);
					\draw (2) to (7);
					\draw (3) to (7);
					\draw (3) to (8);
					\draw (3) to (9);
					\draw (4) to (9);
					\draw (4) to (10);
					\draw (5) to (11);
					\draw (6) to (11);
					\draw (10) to (11);
					\draw (6) to (12);
					\draw (7) to (12);
					\draw (8) to (12);
					\draw (10) to (12);
					\draw (8) to (13);
					\draw (9) to (13);
					\draw (10) to (13);
					\draw [style=blueedge] (3) to (13);
					\draw [style=blueedge] (9) to (8);
					\draw [style=blueedge] (8) to (10);
					\draw [style=blueedge] (13) to (12);
					\draw [style=blueedge] (12) to (11);
					\draw [style=blueedge] (10) to (6);
					\draw [style=blueedge] (10) to (5);
					\draw [style=blueedge] (4) to (11);
					\draw [style=blueedge] (13) to (4);
					\draw [style=blueedge] (9) to (10);
					\draw [style=blueedge] (3) to (4);
					\draw [style=blueedge] (9) to (1);
					\draw [style=blueedge] (6) to (7);
					\draw [style=blueedge] (7) to (8);
					\draw [style=blueedge] (3) to (12);
					\draw [style=blueedge] (6) to (5);
					\draw [style=blueedge] (11) to (2);
					\draw [style=blueedge] (2) to (4);
					\draw [style=blueedge] (1) to (5);
					\draw [style=blueedge, bend left=75, looseness=2.00] (7) to (1);
					\draw [style=blueedge, bend left=75, looseness=2.00] (3) to (2);
					\draw [style=blueedge] (12) to (2);
				\end{pgfonlayer}
			\end{tikzpicture}
		}
		\caption{$G_{13}$}
	\end{subfigure}
	
	\caption{$G_n\in \Maxe{n}{5}$
		for $n= 8, 10, 11$, and $13$.}
	\label{fig:K5free-ex}
\end{figure}
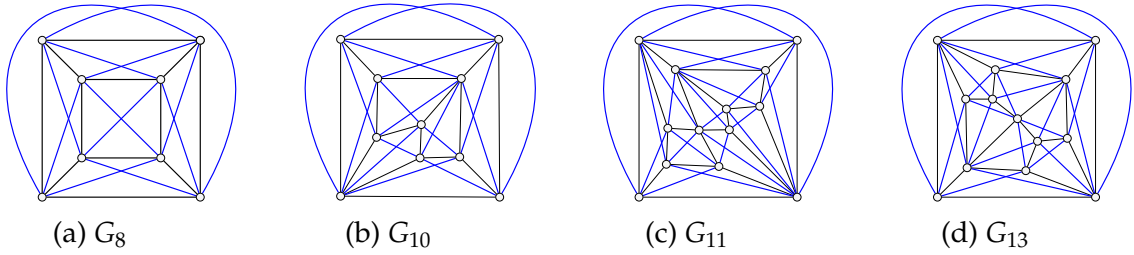

Now we are going to apply the $Q_4$-addition introduced in~\cite{MR2746706}
to create graphs in 
$\Maxe{n}{5}$ 
from  graphs in $\Maxe{n-4}{5}$
for $n=12$ and $n\ge 14$.

Let $G$ be a $1$-plane graph.
We say $G'$ is obtained from 
$G$ by a $Q_4$-addition 
if $G$ contains a clique 
$\{u_1,u_2,u_3, u_4\}$
such that edges $u_1u_3$ 
and $u_2u_4$ 
cross each other in $G$,
and $G'$ is obtained by 
the following operations:
\begin{enumerate}
\item[(i)] remove edges $u_1u_3$ 
and $u_2u_4$ from $G$, 
and obtain a true face $F$
bounded by the cycle
$u_1u_2u_3u_4u_1$ 
in the resulting 
$1$-plane graph $G_0:=G-\{u_1u_3,
u_2u_4\}$;
and 

\item[(ii)] insert 
a $4$-cycle 
$v_1v_2v_3v_4v_1$ within $F$, 
where $v_1,v_2,v_3$ and $v_4$
are new vertices to $G_0$, 
and add new edges 
$u_iv_{i-1}, u_iv_i, u_iv_{i+1}$
for each $i=1,2,3,4$, 
where $v_0$ and $v_5$ 
are actually $v_4$ and $v_1$,
respectively, 
as shown in 
Figure~\ref{fig:add} (b).
\end{enumerate}

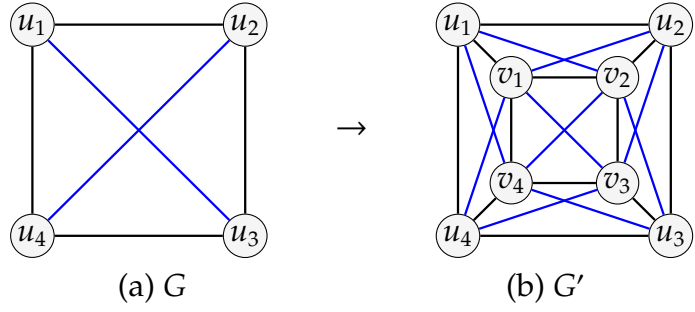
\begin{figure}[h!]
	\centering
	\begin{tikzpicture}[scale=0.7]
		\begin{pgfonlayer}{nodelayer}
			\node [style=whitenode] (12) at (2, 4) {$u_1$};
			\node [style=whitenode] (13) at (6, 4) {$u_2$};
			\node [style=whitenode] (14) at (2, 0) {$u_4$};
			\node [style=whitenode] (15) at (6, 0) {$u_3$};
			\node [style=whitenode] (16) at (10, 4) {$u_1$};
			\node [style=whitenode] (17) at (14, 4) {$u_2$};
			\node [style=whitenode] (18) at (10, 0) {$u_4$};
			\node [style=whitenode] (19) at (14, 0) {$u_3$};
			\node [style=whitenode] (20) at (11, 3) {$v_1$};
			\node [style=whitenode] (21) at (13, 3) {$v_2$};
			\node [style=whitenode] (22) at (11, 1) {$v_4$};
			\node [style=whitenode] (23) at (13, 1) {$v_3$};
			\node [style=none] (25) at (8, 2) {$\rightarrow$};

		\end{pgfonlayer}
		\begin{pgfonlayer}{edgelayer}

			\draw [style=blackedge] (12) to (13);
			\draw [style=blackedge] (13) to (15);
			\draw [style=blackedge] (15) to (14);
			\draw [style=blackedge] (14) to (12);
			\draw [style=blackedge] (16) to (17);
			\draw [style=blackedge] (17) to (19);
			\draw [style=blackedge] (19) to (18);
			\draw [style=blackedge] (18) to (16);
			\draw [style=blackedge] (20) to (21);
			\draw [style=blackedge] (21) to (23);
			\draw [style=blackedge] (23) to (22);
			\draw [style=blackedge] (22) to (20);
			\draw [style=blackedge] (16) to (20);
			\draw [style=blackedge] (22) to (18);
			\draw [style=blackedge] (23) to (19);
			\draw [style=blackedge] (21) to (17);
			\draw [style=blueedge] (13) to (14);
			\draw [style=blueedge] (12) to (15);
			\draw [style=blueedge] (16) to (22);
			\draw [style=blueedge] (18) to (20);
			\draw [style=blueedge] (20) to (17);
			\draw [style=blueedge] (21) to (16);
			\draw [style=blueedge] (22) to (21);
			\draw [style=blueedge] (20) to (23);
			\draw [style=blueedge] (23) to (18);
			\draw [style=blueedge] (22) to (19);
			\draw [style=blueedge] (19) to (21);
			\draw [style=blueedge] (17) to (23);
		\end{pgfonlayer}
	\end{tikzpicture}

(a) $G$ \hspace{4 cm} (b) $G'$
	
	\caption{A new $1$-plane graph obtained 
		by a  $Q_4$-addition
	}
	\label{fig:add}
\end{figure}

\begin{prop}\label{prop:opc}
Let $G$ be a $1$-plane graph.
If $G'$ is obtained from $G$ 
by a $Q_4$-addition, 
then 
\begin{enumerate}
\item[(i)] $G'$ is $1$-plane
with 
$e(G')=e(G)+16$
and 
$G'$ has a subgraph $K_4$
in which two edges cross
each other, 
and 
\item[(ii)] if $G\in \Maxe{n}{5}$, 
then $G'\in \Maxe{n+4}{5}$.
\end{enumerate} 
\end{prop}

\begin{proof}
(i) follows from the definition directly.

(ii). It suffices to show that 
$G'$ is $K_5$-free.
Suppose that $W$ is a $5$-clique in $G'$.
Since $G\in \Maxe{n}{5}$, 
$W\cap \{v_1, v_2, v_3,v_4\}\ne \emptyset$.
Since $N_{G'}(\{v_1, v_2, v_3,v_4\})=\{u_1, u_2, u_3,u_4\}$,
$W$ is a subset of $
\{v_1, v_2, v_3,v_4\}
\cup \{u_1, u_2, u_3,u_4\}$.
Assume that $W=U'\cup V'$,
where $
U'= W\cap \{u_1, u_2, u_3,u_4\}$
and $V'=W\cap \{v_1, v_2, v_3,v_4\}$.
Since $W$ is a clique of $G'$, we have 
$$
|W|\le |U'|
+
\left | 
\bigcap_{u_i\in U'}
\left (N_{G'}(u_i)\cap 
 \{v_1, v_2, v_3,v_4\}\right )
\right |.
$$
However, it can be verified that 
$$
|U'|
+
\left | 
\bigcap_{u_i\in U'}
\left (N_{G'}(u_i)\cap 
 \{v_1, v_2, v_3,v_4\}
 \right )
\right |
=4,
$$
implying that $|W|\le 4$,
a contradiction. 
\end{proof}

We are going to complete the 
proof of Theorem \ref{thm:K5}.

\vspace{3 mm}

\noindent{\it Proof of Theorem \ref{thm:K5}}: 
By Lemma \ref{lem:edges_1planar}
and Propositions~\ref{prop7-1}
and~\ref{prop7-2}, 
the theorem holds 
for $n\in \{1, 2, \cdots, 13\}\setminus \{12\}$.
For $n=12$ or $n\ge 14$, 
by Lemma \ref{lem:edges_1planar} again, 
it suffices to show the existence 
of a graph $G_n$ 
in $\Maxe{n}{5}$ 
with $e(G_n)=4n-8$.

By Propositions~\ref{prop7-2}, 
for $n\in \{8,10, 11,13\}$, 
$G_n$ is the graph defined 
in Figure~\ref{fig:K5free-ex}
with the properties that 
$G_n\in \Maxe{n}{5}$,
$e(G_n)=4n-8$, and 
$G_n$ contains a subgraph 
$K_4$ in which two edges cross
each other.

Now, we define $G_n$ 
for $n=12$ and $n\ge 14$
recursively as follows:
let $G_n$ be the graph 
obtained from $G_{n-4}$ 
by a $Q_4$-addition. 
By Proposition~\ref{prop:opc},
$G_n\in \Maxe{n}{5}$ 
with 
$$
e(G_n)=e(G_{n-4})+16
=4(n-4)-8+16
=4n-8.
$$
Thus, the theorem holds.
\hfill $\Box$

\section{Further study}

Theorem~\ref{thm:K3} shows that for $n\ge 8$, 
$\maxe{n}{3}\le 3n-8$,
where equality holds when $n$ is even.
For odd $n\ge 9$,  it is unknown 
whether $\maxe{n}{3}=3n-8$
holds, 
although Karpov~\cite{zbMATH06347737}
has constructed bipartite graphs $G$ in $\Maxe{n}{3}$ 
with $e(G)= 3n-9$
for all odd $n\ge 5$.
It is natural for us to propose 
the following conjecture 
on $\maxe{n}{3}$ for odd $n$.

\begin{conj}
	\label{conj:odd-tight}
	For any odd $n\ge 5$, 
	$\maxe{n}{3}=3n-9$.
\end{conj}

Even if Conjecture~\ref{conj:odd-tight}
holds, 
 it still remains unknown whether there exists a non-bipartite graph $G \in \Maxe{n}{3}$ with $e(G) = 3n - 8$ for even $n \ge 8$, or $e(G) = 3n - 9$ for odd $n \ge 5$. This naturally leads to the following problem.

\begin{prob}
	\label{conj:bipartite}
Is it true that for any $n\ge 8$,
if $e(G)=\maxe{n}{3}$
for $G\in \Maxe{n}{3}$, 
then  $G$ is bipartite?
\end{prob}


There is a non-bipartite
graph $G\in \Maxe{n}{3}$,
where $n=10$,
such that $e(G)=3n-10$
(see Figure \ref{fig:3n-10}).
However, it is unknown 
if there is a graph 
$G\in \Maxe{n}{3}$,
where $n\ge 10$ is even, 
such that $e(G)=3n-9$.

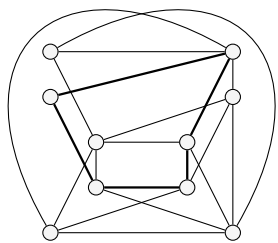
\begin{figure}[H]
\centering
\begin{tikzpicture}[scale=0.6]
	\begin{pgfonlayer}{nodelayer}
		\node [style=whitenode] (0) at (-1, 1) {};
		\node [style=whitenode] (3) at (-2, 0) {};
		\node [style=whitenode] (6) at (-2, 4) {};
		\node [style=whitenode] (8) at (2, 4) {};
		\node [style=whitenode] (11) at (2, 3) {};
		\node [style=whitenode] (12) at (2, 0) {};
		\node [style=whitenode] (16) at (-2, 3) {};
		\node [style=whitenode] (18) at (-1, 2) {};
		\node [style=whitenode] (19) at (1, 2) {};
		\node [style=whitenode] (20) at (1, 1) {};
	\end{pgfonlayer}
	\begin{pgfonlayer}{edgelayer}
		\draw (6) to (8);
		\draw (18) to (19);
		\draw (3) to (12);
		\draw (8) to (11);
		\draw (11) to (12);
		\draw (0) to (18);
		\draw (6) to (18);
		\draw (18) to (3);
		\draw (11) to (18);
		\draw (19) to (12);
		\draw (11) to (20);
		\draw (0) to (12);
		\draw (20) to (3);
		\draw [bend right=75, looseness=2.00] (8) to (3);
		\draw [bend left=75, looseness=2.00] (6) to (12);
		\draw [style=blackedge] (8) to (19);
		\draw [style=blackedge] (19) to (20);
		\draw [style=blackedge] (20) to (0);
		\draw [style=blackedge] (16) to (0);
		\draw [style=blackedge] (8) to (16);
	\end{pgfonlayer}
\end{tikzpicture}

\caption{A non-bipartite
	graph in $\Maxe{n}{3}$ 
	with size $3n-10$ for $n=10$.}
\label{fig:3n-10}
\end{figure}

More generally, it would be interesting to further investigate upper bounds on the number of edges in $1$-planar graphs forbidding other subgraphs. For example, as mentioned by Bekos et al. \cite{Bekos2025}, for $C_k$ with $k\ge 4$, can one obtain (tight) upper bounds? It is interesting  to study the Tur\'an problems on $1$-planar graphs forbidding paths, theta graphs, and other subgraphs.

%
%
%
%
%
%
%
%
%

\section*{Acknowledgements}
We are grateful to Prof.  Pat Morin and Mikhail Kabenyuk for  early discussions with the first author.
The authors declare that they have no conflicts of interest. This research was supported by grants from the National Natural Science Foundation of China (Grant Nos. 12271157, 12371346).

%

\begin{bibdiv}
\begin{biblist}
\bib{MR4010251}{article}{
   author={Ackerman, Eyal},
   title={On topological graphs with at most four crossings per edge},
   journal={Comput. Geom.},
   volume={85},
   date={2019},
   pages={101574, 31},
   issn={0925-7721},
   review={\MR{4010251}},
   doi={10.1016/j.comgeo.2019.101574},
}
\bib{Bekos2025}{article}{
   author={Bekos, Michael A.},
   author={Bose, Prosenjit},
   author={B{\"u}ngener, Alexander},
   author={Dujmovi{\'c}, Vida},
   author={Hoffmann, Michael},
   author={Kaufmann, Michael},
   author={Morin, Pat},
   author={Odak, Svenja},
   author={Weinberger, Angelika},
   title={On $k$-planar graphs without short cycles},
   journal={J. Graph Algorithms Appl.},
   volume={29},
   date={2025},
   pages={1--22},
   issn={1526-1719},
   doi={10.7155/jgaa.v29i3.3003},
}



\bib{MR0732806}{article}{
   author={Bodendiek, R.},
   author={Schumacher, H.},
   author={Wagner, K.},
   title={Bemerkungen zu einem Sechsfarbenproblem von G. Ringel},
   language={German},
   journal={Abh. Math. Sem. Univ. Hamburg},
   volume={53},
   date={1983},
   pages={41--52},
   issn={0025-5858},
   review={\MR{0732806}},
   doi={10.1007/BF02941309},
}
\bib{MR3067240}{article}{
   author={Brandenburg, Franz J.},
   author={Eppstein, David},
   author={Glei\ss ner, Andreas},
   author={Goodrich, Michael T.},
   author={Hanauer, Kathrin},
   author={Reislhuber, Josef},
   title={On the density of maximal 1-planar graphs},
   conference={
      title={Graph drawing},
   },
   book={
      series={Lecture Notes in Comput. Sci.},
      volume={7704},
      publisher={Springer, Heidelberg},
   },
   isbn={978-3-642-36763-2},
   isbn={978-3-642-36762-5},
   date={2013},
   pages={327--338},
   review={\MR{3067240}},
   doi={10.1007/978-3-642-36763-2\_29},
}





\bib{MR2876333}{article}{
   author={Czap, J\'ulius},
   author={Hud\'ak, D\'avid},
   title={1-planarity of complete multipartite graphs},
   journal={Discrete Appl. Math.},
   volume={160},
   date={2012},
   number={4-5},
   pages={505--512},
   issn={0166-218X},
   review={\MR{2876333}},
   doi={10.1016/j.dam.2011.11.014},
}

\bib{MR3084596}{article}{
   author={Czap, J\'ulius},
   author={Hud\'ak, D\'avid},
   title={On drawings and decompositions of 1-planar graphs},
   journal={Electron. J. Combin.},
   volume={20},
   date={2013},
   number={2},
   pages={Paper 54, 8},
   review={\MR{3084596}},
   doi={10.37236/2392},
}



\bib{MR4874150}{book}{
   author={Diestel, Reinhard},
   title={Graph theory},
   series={Graduate Texts in Mathematics},
   volume={173},
   edition={6},
   publisher={Springer, Berlin},
   date={[2025] \copyright 2025},
   pages={xx+454},
   isbn={978-3-662-70106-5},
   isbn={978-3-662-70107-2},
   review={\MR{4874150}},
}

\bib{MR3549506}{article}{
   author={Dowden, Chris},
   title={Extremal $C_4$-free/$C_5$-free planar graphs},
   journal={J. Graph Theory},
   volume={83},
   date={2016},
   number={3},
   pages={213--230},
   issn={0364-9024},
   review={\MR{3549506}},
   doi={10.1002/jgt.21991},
}
\bib{MR0018807}{article}{
   author={Erd\H{o}s, P.},
   author={Stone, A. H.},
   title={On the structure of linear graphs},
   journal={Bull. Amer. Math. Soc.},
   volume={52},
   date={1946},
   pages={1087--1091},
   issn={0002-9904},
   review={\MR{0018807}},
   doi={10.1090/S0002-9904-1946-08715-7},
}

\bib{MR2297168}{article}{
   author={Fabrici, Igor},
   author={Madaras, Tom\'a\v s},
   title={The structure of 1-planar graphs},
   journal={Discrete Math.},
   volume={307},
   date={2007},
   number={7-8},
   pages={854--865},
   issn={0012-365X},
   review={\MR{2297168}},
   doi={10.1016/j.disc.2005.11.056},
}

\bib{MR5037331}{article}{
   author={Gerbner, D\'aniel},
   author={Palmer, Cory},
   title={Survey of Generalized Tur\'an Problems --- Counting Subgraphs},
   journal={Electron. J. Combin.},
   volume={DS27},
   date={2026},
   pages={Paper No. DS27},
   review={\MR{5037331}},
   doi={10.37236/14563},
}

\bib{MR0898434}{book}{
   author={Gross, Jonathan L.},
   author={Tucker, Thomas W.},
   title={Topological graph theory},
   series={Wiley-Interscience Series in Discrete Mathematics and
   Optimization},
   note={A Wiley-Interscience Publication},
   publisher={John Wiley \& Sons, Inc., New York},
   date={1987},
   pages={xvi+351},
   isbn={0-471-04926-3},
   review={\MR{0898434}},
}

\bib{MR4474377}{article}{
   author={Ghosh, Debarun},
   author={Gy\H ori, Ervin},
   author={Martin, Ryan R.},
   author={Paulos, Addisu},
   author={Xiao, Chuanqi},
   title={Planar Tur\'an number of the 6-cycle},
   journal={SIAM J. Discrete Math.},
   volume={36},
   date={2022},
   number={3},
   pages={2028--2050},
   issn={0895-4801},
   review={\MR{4474377}},
   doi={10.1137/21M140657X},
}



\bib{gyori2023}{article}{
   author={Gy\H{o}ri, Ervin},
   author={Li, Alan},
   author={Zhou, Runtian},
   title={The planar Tur\'an number of the seven-cycle},
   journal={arXiv preprint},
   date={2023},
   eprint={2307.06909},
   archivePrefix={arXiv},
   primaryClass={math.CO},
}



\bib{MR2802062}{article}{
   author={Hud\'ak, D\'avid},
   author={Madaras, Tom\'a\v s},
   title={On local properties of 1-planar graphs with high minimum degree},
   journal={Ars Math. Contemp.},
   volume={4},
   date={2011},
   number={2},
   pages={245--254},
   issn={1855-3966},
   review={\MR{2802062}},
   doi={10.26493/1855-3974.131.91c},
}


\bib{zbMATH06347737}{article}{
  author={Karpov, D. V.},
  title={An upper bound on the number of edges in an almost planar bipartite graph},
  journal={J. Math. Sci. (New York)},
  volume={196},
  date={2014},
  number={6},
  pages={737--746},
  issn={1072-3374},
  doi={10.1007/s10958-014-1690-9},
  keywords={05C10, 05C62}
}
\bib{MR4821240}{article}{
   author={Kaufmann, Michael},
   author={Klemz, Boris},
   author={Knorr, Kristin},
   author={Reddy, Meghana M.},
   author={Schr\"oder, Felix},
   author={Ueckerdt, Torsten},
   title={The density formula: one lemma to bound them all},
   conference={
      title={32nd International Symposium on Graph Drawing and Network
      Visualization},
   },
   book={
      series={LIPIcs. Leibniz Int. Proc. Inform.},
      volume={320},
      publisher={Schloss Dagstuhl. Leibniz-Zent. Inform., Wadern},
   },
   isbn={978-3-95977-343-0},
   date={2024},
   pages={Art. No. 7, 17},
   review={\MR{4821240}},
   doi={10.4230/lipics.gd.2024.7},
}

\bib{MR2866732}{article}{
   author={Keevash, Peter},
   title={Hypergraph Tur\'an problems},
   conference={
      title={Surveys in combinatorics 2011},
   },
   book={
      series={London Math. Soc. Lecture Note Ser.},
      volume={392},
      publisher={Cambridge Univ. Press, Cambridge},
   },
   isbn={978-1-107-60109-3},
   date={2011},
   pages={83--139},
   review={\MR{2866732}},

}
%
\bib{MR3990020}{article}{
   author={Lan, Yongxin},
   author={Shi, Yongtang},
   author={Song, Zi-Xia},
   title={Extremal theta-free planar graphs},
   journal={Discrete Math.},
   volume={342},
   date={2019},
   number={12},
   pages={111610, 8},
   issn={0012-365X},
   review={\MR{3990020}},
   doi={10.1016/j.disc.2019.111610},
}






\bib{MR4357025}{article}{
   author={Lan, Yongxin},
   author={Shi, Yongtang},
   author={Song, Zi-Xia},
   title={Planar Tur\'an number and planar anti-Ramsey number of graphs},
   language={English},
   journal={Oper. Res. Trans.},
   volume={25},
   date={2021},
   number={3},
   pages={200--216},
   issn={1007-6093},
   review={\MR{4357025}},
}

\bib{MR3502764}{article}{
   author={Nakamoto, Atsuhiro},
   author={Noguchi, Kenta},
   author={Ozeki, Kenta},
   title={Cyclic 4-colorings of graphs on surfaces},
   journal={J. Graph Theory},
   volume={82},
   date={2016},
   number={3},
   pages={265--278},
   issn={0364-9024},
   review={\MR{3502764}},
   doi={10.1002/jgt.21900},
}



\bib{MR1606052}{article}{
   author={Pach, J\'anos},
   author={T\'oth, G\'eza},
   title={Graphs drawn with few crossings per edge},
   journal={Combinatorica},
   volume={17},
   date={1997},
   number={3},
   pages={427--439},
   issn={0209-9683},
   review={\MR{1606052}},
   doi={10.1007/BF01215922},
}
\bib{MR3751397}{book}{
   author={Schaefer, Marcus},
   title={Crossing numbers of graphs},
   series={Discrete Mathematics and its Applications (Boca Raton)},
   publisher={CRC Press, Boca Raton, FL},
   date={2018},
   pages={xxvi+350},
   isbn={978-1-4987-5049-3},
   review={\MR{3751397}},
}
\bib{MR4828036}{article}{
   author={Shi, Ruilin},
   author={Walsh, Zach},
   author={Yu, Xingxing},
   title={Dense circuit graphs and the planar Tur\'an number of a cycle},
   journal={J. Graph Theory},
   volume={108},
   date={2025},
   number={1},
   pages={27--38},
   issn={0364-9024},
   review={\MR{4828036}},
   doi={10.1002/jgt.23165},
}


\bib{MR4866462}{article}{
   author={Shi, Ruilin},
   author={Walsh, Zach},
   author={Yu, Xingxing},
   title={Planar Tur\'an number of the 7-cycle},
   journal={European J. Combin.},
   volume={126},
   date={2025},
   pages={Paper No. 104134, 24},
   issn={0195-6698},
   review={\MR{4866462}},
   doi={10.1016/j.ejc.2025.104134},
}

\bib{MR2746706}{article}{
   author={Suzuki, Yusuke},
   title={Re-embeddings of maximum 1-planar graphs},
   journal={SIAM J. Discrete Math.},
   volume={24},
   date={2010},
   number={4},
   pages={1527--1540},
   issn={0895-4801},
   review={\MR{2746706}},
   doi={10.1137/090746835},
}
\bib{MR4305146}{article}{
   author={Suzuki, Yusuke},
   title={1-planar graphs},
   conference={
      title={Beyond planar graphs---communications of NII Shonan meetings},
   },
   book={
      publisher={Springer, Singapore},
   },
   isbn={978-981-15-6532-8},
   isbn={978-981-15-6533-5},
   date={[2020] \copyright 2020},
   pages={47--68},
   review={\MR{4305146}},
   doi={10.1007/978-981-15-6533-5\_4},
}



\bib{MR0018405}{article}{
   author={Tur\'an, Paul},
   title={Eine Extremalaufgabe aus der Graphentheorie},
   language={Hungarian, with German summary},
   journal={Mat. Fiz. Lapok},
   volume={48},
   date={1941},
   pages={436--452},
   issn={0302-7317},
   review={\MR{0018405}},
}


\bib{XuChang25}{article}{
   author={Xu, Weilun},
   author={Chang, An},
   title={The spectral radius of $1$-planar graphs without complete subgraphs},
   journal={arXiv preprint},
   date={2025},
   eprint={2512.12909},
   archivePrefix={arXiv},
   primaryClass={math.CO},
}


\bib{MR4218488}{article}{
   author={Zhang, Xin},
   author={Li, Yan},
   title={Dynamic list coloring of 1-planar graphs},
   journal={Discrete Math.},
   volume={344},
   date={2021},
   number={5},
   pages={Paper No. 112333, 8},
   issn={0012-365X},
   review={\MR{4218488}},
   doi={10.1016/j.disc.2021.112333},
}

\end{biblist}
\end{bibdiv}
\end{document}